\documentclass[12pt]{amsart}
\usepackage{verbatim}
\usepackage{amssymb, amsmath}
\usepackage{graphicx}
\usepackage{caption}
\usepackage{subcaption}
\usepackage{color}

\begin{document}
\newtheorem{thm}{Theorem}[section]
\newtheorem*{thm*}{Theorem}
\newtheorem{lem}[thm]{Lemma}
\newtheorem{prop}[thm]{Proposition}
\newtheorem{cor}[thm]{Corollary}
\newtheorem*{conj}{Conjecture}
\newtheorem{proj}[thm]{Project}
\newtheorem{question}[thm]{Question}
\newtheorem{rem}{Remark}[section]

\theoremstyle{definition}
\newtheorem*{defn}{Definition}
\newtheorem*{remark}{Remark}
\newtheorem{exercise}{Exercise}
\newtheorem*{exercise*}{Exercise}

\numberwithin{equation}{section}

\newcommand{\rad}{\operatorname{rad}}

\newcommand{\Z}{{\mathbb Z}} 
\newcommand{\Q}{{\mathbb Q}}
\newcommand{\R}{{\mathbb R}}
\newcommand{\C}{{\mathbb C}}
\newcommand{\N}{{\mathbb N}}
\newcommand{\FF}{{\mathbb F}}
\newcommand{\fq}{\mathbb{F}_q}
\newcommand{\rmk}[1]{\footnote{{\bf Comment:} #1}}

\renewcommand{\mod}{\;\operatorname{mod}}
\newcommand{\ord}{\operatorname{ord}}
\newcommand{\TT}{\mathbb{T}}
\renewcommand{\i}{{\mathrm{i}}}
\renewcommand{\d}{{\mathrm{d}}}
\renewcommand{\^}{\widehat}
\newcommand{\HH}{\mathbb H}
\newcommand{\Vol}{\operatorname{vol}}
\newcommand{\area}{\operatorname{area}}
\newcommand{\tr}{\operatorname{tr}}
\newcommand{\norm}{\mathcal N} 
\newcommand{\intinf}{\int_{-\infty}^\infty}
\newcommand{\ave}[1]{\left\langle#1\right\rangle} 
\newcommand{\Var}{\operatorname{Var}}
\newcommand{\Prob}{\operatorname{Prob}}
\newcommand{\sym}{\operatorname{Sym}}
\newcommand{\disc}{\operatorname{disc}}
\newcommand{\CA}{{\mathcal C}_A}
\newcommand{\cond}{\operatorname{cond}} 
\newcommand{\lcm}{\operatorname{lcm}}
\newcommand{\Kl}{\operatorname{Kl}} 
\newcommand{\leg}[2]{\left( \frac{#1}{#2} \right)}  
\newcommand{\Li}{\operatorname{Li}}

\newcommand{\sumstar}{\sideset \and^{*} \to \sum}

\newcommand{\LL}{\mathcal L} 
\newcommand{\sumf}{\sum^\flat}
\newcommand{\Hgev}{\mathcal H_{2g+2,q}}
\newcommand{\USp}{\operatorname{USp}}
\newcommand{\conv}{*}
\newcommand{\dist} {\operatorname{dist}}
\newcommand{\CF}{c_0} 
\newcommand{\kerp}{\mathcal K}

\newcommand{\Cov}{\operatorname{cov}}
\newcommand{\Sym}{\operatorname{Sym}}

\newcommand{\Ht}{\operatorname{Ht}}

\newcommand{\E}{\operatorname{\mathbb E}} 
\newcommand{\sign}{\operatorname{sign}} 
\newcommand{\meas}{\operatorname{meas}} 
\newcommand{\length}{\operatorname{length}} 

\newcommand{\divid}{d} 

\newcommand{\GL}{\operatorname{GL}}
\newcommand{\SL}{\operatorname{SL}}
\newcommand{\re}{\operatorname{Re}}
\newcommand{\im}{\operatorname{Im}}
\newcommand{\res}{\operatorname{Res}}
 \newcommand{\eigen}{\Lambda} 
\newcommand{\tens}{\mathbf t} 
\newcommand{\diam}{\operatorname{diam}}

\newcommand{\new}[1]{{\color{blue} #1}}
\newcommand{\fixme}[1]{\footnote{Fixme: #1}}

\title{Differences between Robin and Neumann eigenvalues}
\author{Ze\'ev Rudnick, Igor Wigman and Nadav Yesha}

\address{School of Mathematical Sciences, Tel Aviv University, Tel Aviv 69978, Israel} \email{rudnick@tauex.tau.ac.il}
\address{Department of Mathematics, King's College London, UK}\email{igor.wigman@kcl.ac.uk}
\address{Department of Mathematics, University of Haifa, 3498838 Haifa, Israel}\email{nyesha@univ.haifa.ac.il}

\thanks{We thank Michael Levitin and Iosif Polterovich for the their comments. This research was supported by the European Research Council (ERC) under the European Union's  Horizon 2020 research and innovation programme  (Grant agreement No.    786758) and by the ISRAEL SCIENCE FOUNDATION
(grant No. 1881/20).}

\begin{abstract}
Let $\Omega\subset \R^2$ be a bounded planar domain, with piecewise smooth boundary $\partial \Omega$. For $\sigma>0$, we consider the Robin boundary value problem
\[
-\Delta f =\lambda f, \qquad \frac{\partial f}{\partial n} + \sigma f = 0 \mbox{ on } \partial \Omega
\]
 where $ \frac{\partial f}{\partial n} $ is the derivative in the direction of the outward pointing normal to $\partial \Omega$.
 Let $0<\lambda^\sigma_0\leq  \lambda^\sigma_1\leq \dots $ be the corresponding eigenvalues.  The purpose of this paper is to study the Robin-Neumann gaps
 \[
d_n(\sigma):=\lambda_n^\sigma-\lambda_n^0 .
\]
For a wide class of planar domains we show that there is a limiting mean value, equal to $2\length(\partial\Omega)/\area(\Omega)\cdot \sigma$ and  in the smooth case, give an upper bound of $d_n(\sigma)\leq C(\Omega ) n^{1/3}\sigma $ and a uniform lower bound.
For ergodic billiards  we show that along a density-one subsequence,  the gaps converge to the mean value.
We obtain further properties for rectangles,
where we have a uniform upper bound, and for disks, where we improve the general upper bound.
 \end{abstract}
  \keywords{ Robin boundary conditions, Robin-Neumann gaps, Laplacian, ergodic billiards, quantum ergodicity, lattice point problems}
  \subjclass[2010]{Primary 35P20,  Secondary 37D50 , 58J51, 81Q50, 11P21 }
\date{\today}
\maketitle

\tableofcontents

\section{Statement of results}

Let $\Omega\subset \R^2$ be a bounded planar domain, with piecewise smooth boundary $\partial \Omega$. For $\sigma\geq 0$, we consider the Robin boundary value problem
\[
-\Delta f =\lambda f \;{\rm on}\;\Omega, \quad
\frac{\partial f}{\partial n} + \sigma f = 0 \mbox{ on } \partial \Omega
\]
 where $ \frac{\partial f}{\partial n} $ is the derivative in the direction of the outward pointing normal to $\partial \Omega$.
The case $\sigma=0$ is the Neumann boundary condition, and we use $\sigma=\infty$ as a shorthand for the Dirichlet boundary condition $f|_{\partial \Omega}=0$.

Robin boundary conditions are used in heat conductance theory to interpolate between a perfectly insulating boundary, described by  Neumann boundary conditions $\sigma=0$, and a temperature fixing boundary, described by Dirichlet
boundary conditions corresponding to $\sigma=+\infty$. To date, most studies concentrated on the first few Robin eigenvalues, with applications in shape optimization and related isoperimetric inequalities and asymptotics of the first eigenvalues   (see \cite{BFK}).
Our goal is very different, aiming to study the difference between  high-lying Robin and Neumann eigenvalues. There are very few studies addressing this in the literature, except for \cite{Sieber}, \cite{BerryDennis} which aim at different goals.

We will take the Robin condition for a fixed and positive
$\sigma>0$, when all eigenvalues are positive, one excuse being that a negative Robin parameter  gives non-physical boundary conditions for the heat equation, with heat flowing from cold to hot; see however \cite{LOS} for a model where negative $\sigma$ is of interest, in particular  $\sigma\to -\infty$  \cite{LZ, LP, DK, Khalile}.
Let $0<\lambda^\sigma_0\leq  \lambda^\sigma_1\leq \dots $ be the corresponding eigenvalues.
The Robin spectrum always lies between the Neumann and Dirichlet spectra (Dirichlet-Neumann bracketing) \cite{BFK} :
\begin{equation}\label{NRD bracketing}
\lambda_n^0< \lambda_n^\sigma <  \lambda_n^\infty .
\end{equation}
We define the Robin-Neumann difference (RN gaps) as
\[
d_n(\sigma):=\lambda_n^\sigma-\lambda_n^0
\]
and study several of their properties. See \S\ref{sec:numerics} for some numerical experiments. This seems to be a novel subject, and the only related study that we are aware of is the very recent work of  Rivi\`ere and  Royer \cite{RR}, which addresses the RN gaps for quantum star graphs.

\subsection{The mean value}
The first result concerns the mean value of the gaps:
\begin{thm}\label{thm:mean of E RN gaps}
Let $\Omega\subset \R^2$ be a bounded, piecewise smooth domain.   Then the mean value of the RN gaps exists, and equals
$$\lim_{N\to \infty} \frac 1N \sum_{n=1}^N d_n(\sigma) = \frac{2\length(\partial \Omega)}{\area (\Omega)}\cdot \sigma .
$$
\end{thm}
  Since the differences $d_n(\sigma)> 0$ are positive, we deduce by Chebyshev's inequality:
\begin{cor}\label{cor: gaps grow slowly}
Let $\Omega$ be a bounded, piecewise smooth domain. Fix $\sigma>0$. Let $\Phi(n)\to \infty$ be a function tending to infinity (arbitrarily slowly). Then for almost all $n$'s, $d_n(\sigma) \leq \Phi(n)$
in the sense that
\[
\#\{n\leq N  :  d_n(\sigma)> \Phi(n) \} \ll \frac{N}{\Phi(N)} .
\]
\end{cor}

\subsection{A lower bound}

Recall that a domain $\Omega $ is ``star-shaped with respect to   a point $x\in \Omega$"  if  the segment between $x$ and every other point of $\Omega$ lies inside the domain; so convex means star-shaped with respect to  any point;
``star-shaped" just means that there is some $x$ so that it is star-shaped with respect to   $x$.

\begin{thm}\label{thm:RN lower bound}
Let $\Omega\subset \R^2$ be a bounded star-shaped planar domain with 
smooth boundary. Then the Robin-Neumann differences are uniformly bounded below: For all $\sigma>0$, $\exists C= C(\Omega,\sigma)>0$ so that
\[
d_n(\sigma) \geq C  .
\]
\end{thm}

Note that for quantum star graphs, this lower bound fails \cite{RR}. 

\subsection{A general upper bound}

 We give a quantitative upper bound:
\begin{thm}\label{general smooth upper bd}
Assume that $\Omega$ has a smooth boundary. Then $\exists C= C(\Omega )>0$ so that for all $\sigma>0$,
  \begin{equation*}
 d_n(\sigma)\leq C  (\lambda_n^\infty)^{1/3} \sigma .
 \end{equation*}
 \end{thm}
 While quite poor, it is the best  individual bound that we have in general. 
 Below, we will indicate how to improve it in special cases.

\begin{question}
Are there planar domains where the differences $d_n(\sigma)$ are {\bf unbounded}?
\end{question}
We believe that this happens in several cases, e.g. the disk, but at present can only show this for the hemisphere \cite{RW}, which is not a planar domain.

\subsection{Ergodic billiards}
To a piecewise smooth planar domain one associates a billiard dynamics. When this dynamics is ergodic, as for the stadium billiard
(see Figure~\ref{fig:ergodic gaps}), we can improve on Corollary~\ref{cor: gaps grow slowly}:
\begin{thm}\label{thm:QE RN gaps}
Let $\Omega\subset \R^2$ be a bounded, piecewise smooth domain. Assume that the billiard dynamics associated to $\Omega$ is ergodic. Then for every $\sigma>0$, there is a sub-sequence  $\mathcal N=\mathcal N_\sigma \subset \mathbb N$ of density one  so that along that subsequence,
$$\lim_{\substack{n\to \infty \\ n\in \mathcal N}}d_n(\sigma)= \frac{2\length(\partial \Omega)}{\area (\Omega)}\cdot \sigma .
$$
\end{thm}

 If the billiard dynamics is uniformly hyperbolic, we expect that more is true,  that {\em all} the gaps  converge to the mean.

A key ingredient in the proofs of the above results is that they can be connected to   $L^2$ restriction estimates for eigenfunctions on the boundary via a variational formula for the gaps (Lemma~\ref{lem:formula for dn})
 \begin{equation*}
d_n(\sigma) =  \int_0^\sigma \left( \int_{\partial \Omega} |u_{n,\tau}|^2 ds  \right) d\tau
\end{equation*}
where $u_{n,\tau}$ is any $L^2(\Omega)$-normalized eigenfunction associated with $\lambda_n^\tau$.

\subsection{Generalizations}
Most of the above results easily extend to higher dimensions: The upper bound (Theorem~\ref{general smooth upper bd}), the mean value result (Theorem~\ref{thm:mean of E RN gaps}), and the almost sure convergence for ergodic billiards (Theorem~\ref{thm:QE RN gaps}). At this stage our proof of the lower bound (Theorem~\ref{thm:RN lower bound}) is restricted to dimension $2$. 

In section~\ref{sec:Variable sigma} we discuss extensions of the above results to the case of variable boundary conditions $\sigma: \partial \Omega\to \R$.

\subsection{Rectangles} 
 For the special case of rectangles, we show that the RN gaps are bounded:
\begin{thm}\label{thm:uniform upper bound}
Let $\Omega$ be a rectangle. Then for every $\sigma>0$  there is some $C_\Omega(\sigma)>0$ so that for all  $n$,
$$d_n(\sigma) \leq C_\Omega(\sigma).
$$
\end{thm}

We use Theorem~\ref{thm:uniform upper bound} 
to draw a consequence for the level spacing distribution of the Robin eigenvalues on a rectangle: 
 Let $x_0\leq x_1\leq x_2\leq \dots$ be a sequence of levels, and $\delta_n=x_{n+1}-x_n$ be the nearest neighbour gaps. We assume that $x_N=N+o(N)$ so that the average gap is unity. The level spacing distribution $P(s)$ of the sequence is then defined as
\[
\int_0^y P(s)ds = \lim_{N\to \infty} \frac 1N \#\{n\leq N: \delta_n\leq y\}
\]
(assuming that the limit exists).

It is well known that the level spacing distribution for the Neumann (or Dirichlet) eigenvalues on the square is a delta-function at the origin, due to large arithmetic multiplicities in the spectrum. Once we put a Robin boundary condition, we can show \cite{RWrectangle}
that the multiplicities disappear for $\sigma>0$ sufficiently small, except for systematic doubling due to symmetry. 
Nonetheless, even after desymmetrizing (removing the systematic multiplicities) we show that the level spacing does not change:
\begin{thm}
 The level spacing distribution for the desymmetrized
 Robin spectrum on the square  is a delta-function at the origin.
\end{thm}


\subsection{The disk}
As we will explain, upper bounds for the gaps $d_n$ can be obtained from upper bounds for the remainder term in Weyl's law for the Robin/Neumann problem.
While this method will usually fall short of Theorem~\ref{general smooth upper bd}, for the disk it gives a better bound.
In that case,  Kuznetsov and Fedosov \cite{KuzFed65} (see also  Colin de Verdi\'ere  \cite{CdV disk}) gave an improved remainder term in Weyl's law for Dirichlet boundary conditions,
by relating the problem to  counting (shifted)  lattice points in a certain cusped domain.
 With some work, the argument  can also be adapted to the Robin case
 (see \S~\ref{sec:weyllatdisk} and Appendix~\ref{sec:BesselAppendix}),  which recovers Theorem~\ref{general smooth upper bd} in this special case.
 The remainder term for the lattice count  was  improved by Guo, Wang and  Wang \cite{GWW},
 from which we  obtain:
 \begin{thm}\label{thm:upper bd for disk}
For the unit disk,   for any fixed $\sigma>0$, we have
\[
d_n(\sigma)  =O(n^{1/3-\delta}), \quad \delta=1/990.
\]
 \end{thm}

\section{Numerics}\label{sec:numerics}

We present some numerical experiments on the fluctuation of the RN gaps. In all cases, we took the Robin constant to be $\sigma=1$. Displayed are the run sequence plots  of the RN gaps.  The solid (green) curve is the cumulative mean. The solid (red) horizontal line is the limiting mean value $2\length(\partial \Omega)/\area(\Omega) $
obtained in Theorem~\ref{thm:mean of E RN gaps}.

In Figure~\ref{circle and square gaps} we present numerics for two domains where the Neumann and Dirichlet problems are solvable, by means of separation of variables, the square and the disk. These were generated using Mathematica \cite{Mathematica}. For the square, we are reduced to finding Robin eigenvalues on an interval as (numerical) solutions to a secular equation,
see \S~\ref{sec:rectangles}, and have used Mathematica to find these.
   \begin{figure}[ht]
    \begin{subfigure}[b]{0.45\textwidth}
    \includegraphics[width=\textwidth]{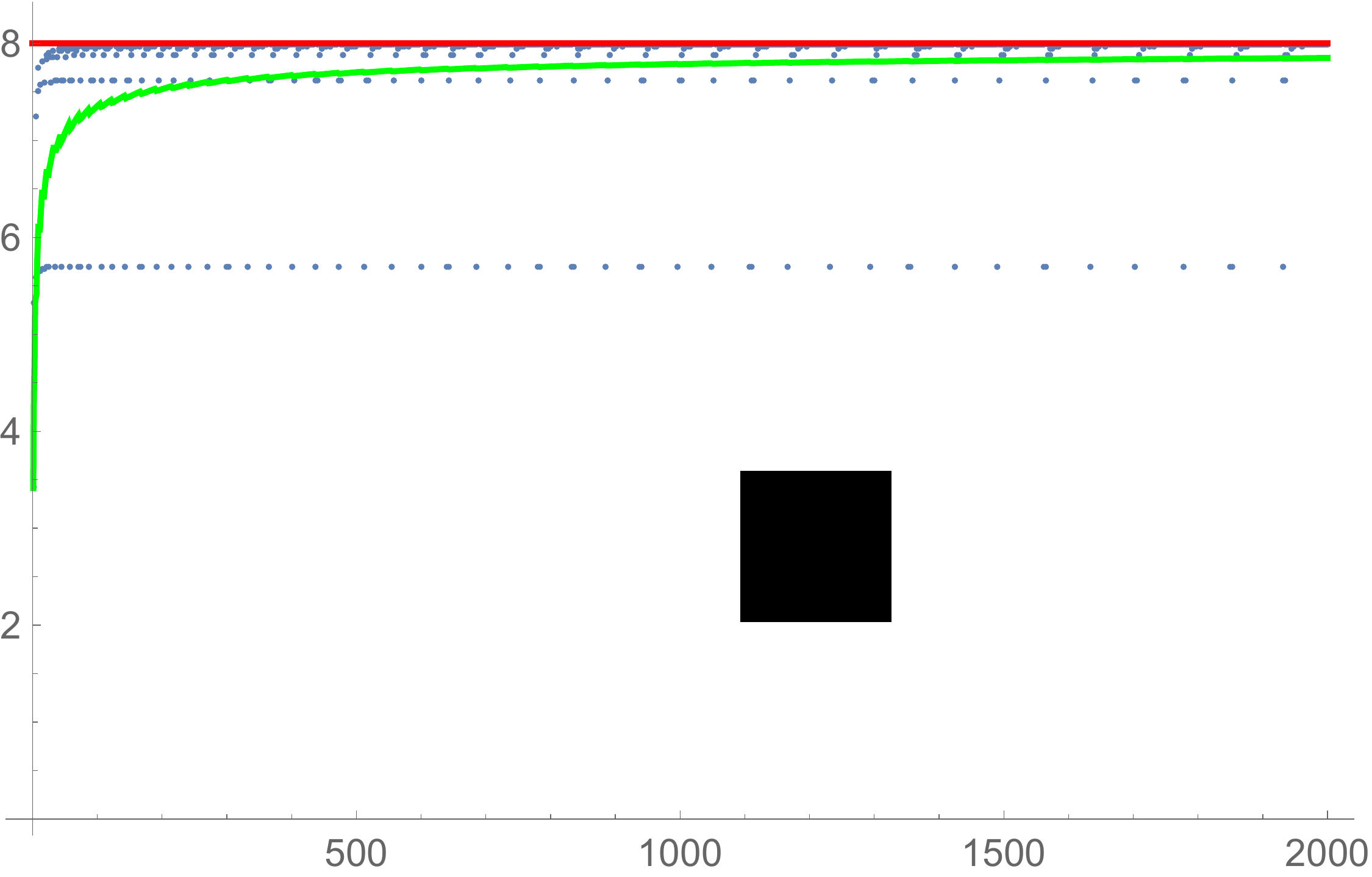}
    \subcaption{The unit square.\label{figsub:square}}

  \end{subfigure}
    \hfill
  \begin{subfigure}[b]{0.45\textwidth}
    \includegraphics[width=\textwidth]{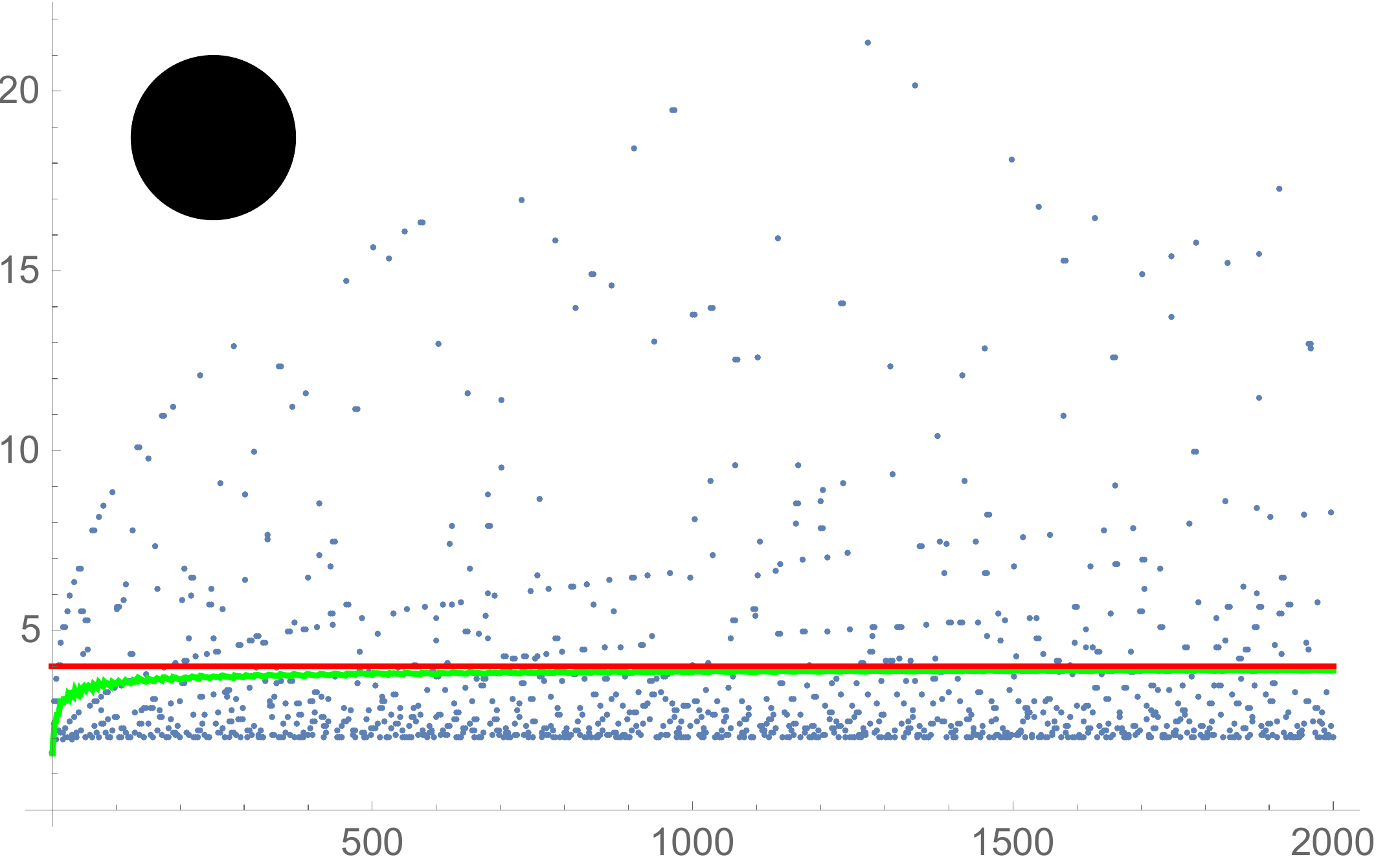}
    \subcaption{The unit disk.\label{figsub:circle}}

  \end{subfigure}
  \caption{The first $2000$ RN gaps  for unit square \subref{figsub:square} and for the unit disk \subref{figsub:circle}.
    }
   \label{circle and square gaps}
\end{figure}


The disk admits separation of variables, and as is well known the Dirichlet eigenvalues on the unit disk are the squares of the
  positive zeros of the Bessel functions $J_n(x)$. The positive Neumann eigenvalues are squares of the positive zeros   of the derivatives  $J_n'(x)$, and the Robin eigenvalues are the squares of the positive zeros of $xJ_n'( x) + \sigma J_n(x)$. We generated these using Mathematica,  see Figure~\ref{figsub:circle}.

 \begin{figure}[ht]
  \begin{subfigure}[b]{0.45\textwidth}
    \includegraphics[width=\textwidth]{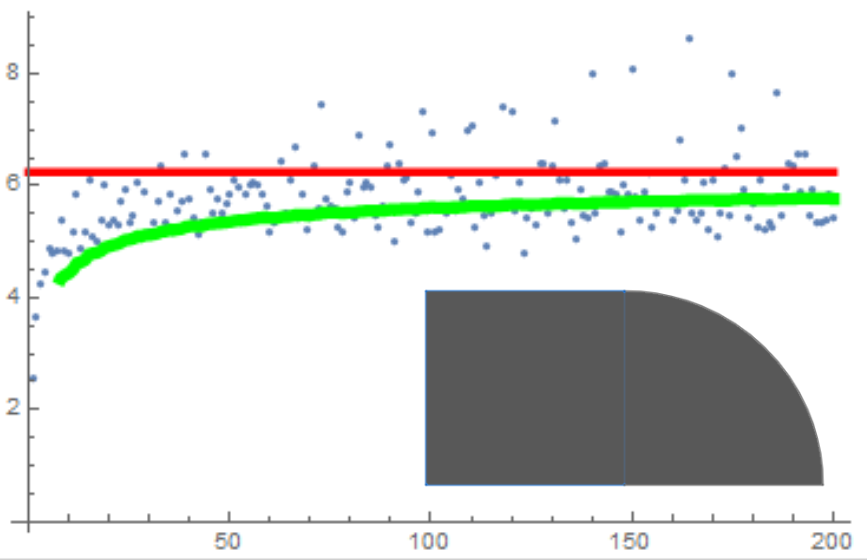}
    \subcaption{The stadium billiard.    \label{figsub:bunimovich}}
  \end{subfigure}
  \hfill
  \begin{subfigure}[b]{0.45\textwidth}
    \includegraphics[width=\textwidth]{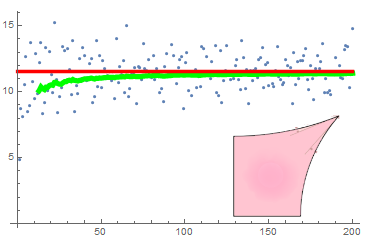}
    \subcaption{A dispersing billiard.    \label{figsub:barnett}}
  \end{subfigure}
  \caption{The first  200 RN gaps   for the ergodic quarter-stadium billiard \subref{figsub:bunimovich}, a quarter of the shape formed by  gluing two half-disks to a square of sidelength $2$, and for the uniformly hyperbolic billiard consisting of  a quarter of the shape formed by the intersection of the exteriors of four disks \subref{figsub:barnett}.}
     \label{fig:ergodic gaps}
\end{figure}
For the remaining cases we used the finite elements package FreeFem \cite{FreeFem}, \cite{LevitinFreeFem}.
In Figure~\ref{fig:ergodic gaps} we display two ergodic examples, the quarter-stadium billiard and a uniformly hyperbolic, Sinai-type dispersing billiard which was investigated numerically by Barnett \cite{BarnettQE}.

It is also of interest to understand rational polygons, that is simple plane polygons all of whose vertex angles are rational multiples of $\pi$  (Figure~\ref{fig:pseudointegrable gaps}), when we expect an analogue of Theorem~\ref{thm:QE RN gaps} to hold, compare \cite{MR}.
 \begin{figure}[ht]
  \begin{subfigure}[b]{0.45\textwidth}
    \includegraphics[width=\textwidth]{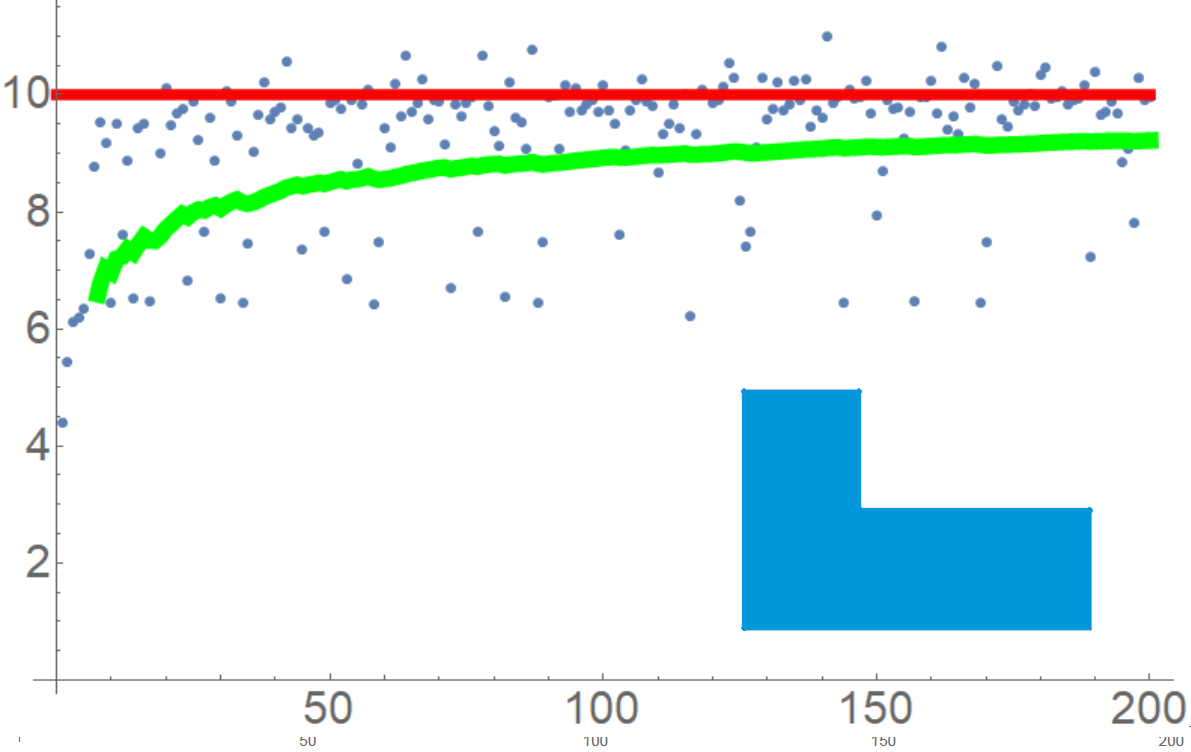}
    \subcaption{An L-shaped billiard.     \label{figsub:Lshaped}}
  \end{subfigure}
  \hfill
  \begin{subfigure}[b]{0.45\textwidth}
    \includegraphics[width=\textwidth]{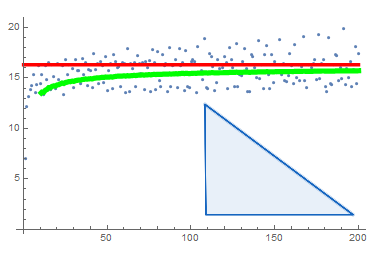}
    \subcaption{Right triangle with angle $\pi/5$.    \label{figsub:pi5triangle}}
  \end{subfigure}
  \caption{The first  $200$ RN gaps   for two examples of rational polygons:
  An L-shaped billiard \subref{figsub:Lshaped} made of $4$ squares of sidelength $1/2$,  and a right triangle with an angle $\pi/5$
  and a long side of length unity  \subref{figsub:pi5triangle}.
   }
   \label{fig:pseudointegrable gaps}
\end{figure}
The case of dynamics with a mixed phase space, such as the mushroom billiard investigated by
Bunimovich \cite{Bunimovich mushroom} (see also the survey \cite{PL})  also deserves study, see  Figure~\ref{fig:mushroom}.
\begin{figure}[ht]
  \begin{subfigure}[b]{0.3\textwidth}
    \includegraphics[width=\textwidth]{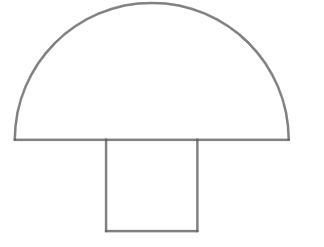}
    \subcaption{The mushroom billiard.}
    \label{figsub:mushroompic}
  \end{subfigure}
  \hfill
  \begin{subfigure}[b]{0.5\textwidth}
    \includegraphics[width=\textwidth]{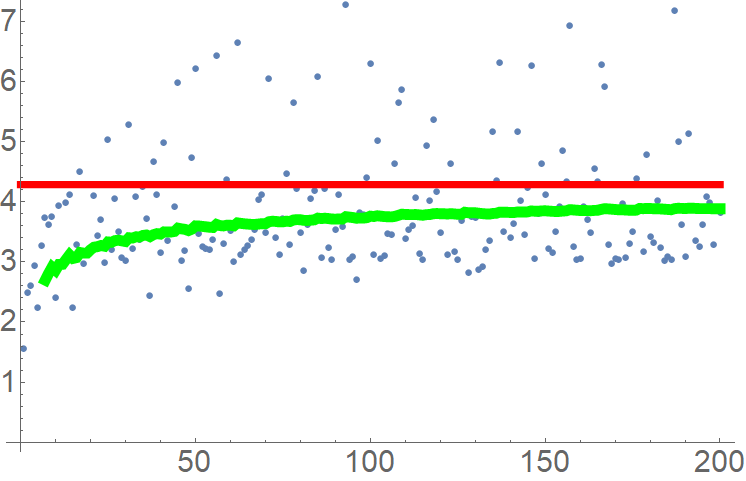}
    \subcaption{The first 200 RN gaps}
    \label{figsub:mushroom200}
  \end{subfigure}
  \caption{The first  $200$ RN gaps   for the mushroom billiard, with a half-disk of diameter 3 on top of a unit square, which has mixed (chaotic and regular) billiard dynamics.}
  \label{fig:mushroom}
\end{figure}


\section{Generalities about the RN gaps}

\subsection{Robin-Neumann bracketing and positivity of the RN gaps}

We recall the min-max characterization of the Robin eigenvalues
\[
\lambda_n^\sigma = \inf_{\substack{M\subset H^1(\Omega) \\\dim M=n}} \sup_{0\neq u\in M} \frac{\int_\Omega |\nabla u|^2 dx +\int_{\partial\Omega} \sigma u^2 ds}{\int_\Omega u^2 dx}
\]
where $H^1(\Omega)$ is the Sobolev space. This shows that $\lambda_n^\sigma\geq \lambda_n^0$ if $\sigma>0$. Likewise, there is a min-max characterization of the Dirichlet eigenvalues with $H^1(\Omega)$ replaced by the subspace $H^1_0(\Omega)$, the closure  of functions vanishing near the boundary:
\[
\lambda_n^\infty = \inf_{\substack{M\subset H^1_0(\Omega) \\\dim M=n}} \sup_{0\neq u\in M} \frac{\int_\Omega |\nabla u|^2 dx }{\int_\Omega u^2 dx} .
\]
This shows that $\lambda_n^\sigma\leq \lambda_n^\infty$.  

In fact, we have strict inequality,  
\[
\lambda_n^0<\lambda_n^\sigma<\lambda_n^\infty.
\] 
This is proved (in greater generality) in \cite{Rohleder} using a unique continuation principle. 

\subsection{A variational formula for the gaps}

\begin{lem}\label{lem:formula for dn}
 Let  $\Omega\subset \R^d$ be a bounded Lipschitz domain. Then
\begin{equation}\label{formula for dn}
d_n(\sigma):=\lambda_n^\sigma-\lambda_n^0 =  \int_0^\sigma \left( \int_{\partial \Omega} |u_{n,\tau}|^2 ds  \right) d\tau
\end{equation}
where $u_{n,\tau}$ is any $L^2(\Omega)$-normalized eigenfunction associated with $\lambda_n^\tau$.
\end{lem}
\begin{proof}
According to \cite[Lemma 2.11]{AFK} (who attribute it as folklore), for any bounded Lipschitz domain $\Omega\subset \R^d$, and $n\geq 1$, the function $\sigma\to \lambda_n^\sigma$ is strictly increasing for $\sigma\in [0,\infty)$,
is differentiable almost everywhere in $(0,\infty)$, is piecewise analytic, and the non-smooth points are locally finite (i.e. finite in each bounded interval). It is absolutely continuous, and in particular its derivative $d\lambda_n^\sigma/d\sigma$
(which exists almost everywhere) is locally integrable, and  for any $0\leq \alpha<\beta$,
\[
\lambda_n^\beta - \lambda_n^\alpha = \int_\alpha^\beta \frac{d\lambda_n^\sigma}{d\sigma}d\sigma .
\]
Moreover, there is a variational formula valid at any point where the derivative exists:
\begin{equation}\label{Var formula}
\frac{d \lambda_{n}^\sigma}{d\sigma} = \int_{\partial \Omega} |u_{n,\sigma}|^2 ds  
\end{equation}
where $u_{n,\sigma}$ is any normalized eigenfunction associated with $\lambda_n^\sigma$.
Therefore
\begin{equation*}
d_n(\sigma)=\lambda_n^\sigma-\lambda_n^0 = \int_0^\sigma \frac{d \lambda_{n}^\tau}{d\tau}  d\tau = \int_0^\sigma \left( \int_{\partial \Omega} |u_{n,\tau}|^2 ds  \right) d\tau .
\end{equation*}
We can ignore the finitely many points $\tau$ where \eqref{Var formula} fails, as the derivative is integrable.
\end{proof}

\subsection{A general upper bound: Proof of Theorem~\ref{general smooth upper bd} }
 As a corollary, we can show that   for the case of smooth boundary, we have an upper bound\footnote{Here and in the sequel we use $A\ll B$ as an alternative to $A=O(B)$.}
  \begin{equation*}
 d_n(\sigma)\ll_{\Omega,\sigma} (\lambda_n^\infty)^{1/3}  .
 \end{equation*}
 Indeed, for the case of smooth boundary, \cite[Proposition 2.4]{BHT}\footnote{Their Proposition 2.4 is stated only
 for the Neumann case,
 but as is pointed out in Remark 2.7, the proof applies to Robin case as well, uniformly in $\sigma\geq 0$; and they attribute it to Tataru \cite[Theorem 3]{Tataru}. Note that their convention for the normal derivative   is different than ours.}
 give an upper bound on the boundary integrals of eigenfunctions
 $$\int_{\partial \Omega} u_{n,\sigma}^2 ds\ll_{\Omega } (\lambda_n^\sigma)^{1/3} \leq (\lambda_n^\infty)^{1/3},
 $$
 uniformly in $\sigma\geq 0$.
 As a consequence of the variational formula \eqref{formula for dn}, we deduce
 \[
 d_n(\sigma)\ll_{\Omega } (\lambda_n^\infty)^{1/3} \cdot \sigma 
 \]
 and in particular for planar domains, using Weyl's law,  we obtain for $n\geq 1$
\[
 d_n(\sigma)\ll_{\Omega } n^{1/3}\cdot \sigma  .
 \]

\section{The mean value}\label{sec:mean value}
In this section we give a proof of Theorem~\ref{thm:mean of E RN gaps}, that
\[
 \lim_{N\to \infty}  \frac 1N\sum_{n\leq N}   d_n(\sigma)  
 = \frac{2\length(\partial \Omega)}{\area (\Omega)}  \sigma .
\]

Denote
\[
W_N(\sigma):=  \frac 1N\sum_{n\leq N}   \int_{\partial \Omega} u_{n,\sigma}^2 ds .
\]
Using Lemma~\ref{lem:formula for dn} gives
\[
 \frac 1N \sum_{n=1}^N d_n(\sigma) = \int_0^\sigma \left(  \frac 1N\sum_{n\leq N}   \int_{\partial \Omega} u_{n,\tau}^2 ds\right) d\tau =  \int_0^\sigma W_N(\tau) d\tau .
\]
 The local Weyl law \cite{HZ} (valid for any piecewise smooth $\Omega$) shows that for any fixed $\sigma$,
 \[
 \lim_{N\to \infty} W_N(\sigma)  = \frac{2\length(\partial \Omega)}{\area (\Omega)}
 \]
  so if we know that $W_N(\tau)\leq C$ is uniformly bounded for all $\tau\leq \sigma$,
then by the Dominated Convergence Theorem we  deduce that
 \begin{multline*}
 \lim_{N\to \infty}  \frac 1N \sum_{n=1}^N d_n(\sigma)  =  \int_0^\sigma \lim_{N\to \infty} W_N(\tau)  d\tau
  \\
   = \int_0^\sigma  \frac{2\length(\partial \Omega)}{\area (\Omega)}  d\tau =  \frac{2\length(\partial \Omega)}{\area (\Omega)}  \cdot \sigma
 \end{multline*}
 as claimed.

It remains to prove a 
uniform upper bound for $W_N(\sigma)$.
 \begin{lem}\label{uniform local weyl law}
 There is a constant $C=C(\Omega)$ so that for all $\sigma>0$ and all $N\geq 1$,
 \[
 \frac 1N\sum_{n\leq N}   \int_{\partial \Omega} u_{n,\sigma}^2 ds  \leq C . 
 \]
 \end{lem}

 \begin{proof}
 What we use is an upper bound on the heat kernel  {\em on the boundary}. Let $K_\sigma(x,y;t)$ be the heat kernel for the Robin problem. Then  \cite[Lemma 12.1]{HZ},
 \begin{equation}\label{upper bd for Robin kernel}
 K_\sigma(x,y;t)\leq C t^{-\dim \Omega/2}\exp(-\delta| x-y|^2/t)
 \end{equation}
  where $C,\delta>0$ depend only on the domain $\Omega$. Moreover, on the regular part of the boundary,
 \[
 K_\sigma(x,y;t) = \sum_{n\geq 0}e^{-t\lambda_n^\sigma} u_{n,\sigma}(x) u_{n,\sigma}(y) .
 \]

We   have for $n\leq N$ that $\lambda_n^\sigma\leq \lambda_N^\sigma\leq \lambda_N^\infty$, so  for $\Lambda=\lambda_N^\infty$,
\[
\sum_{n\leq N} \int_{\partial \Omega} u_{n,\sigma}^2 ds \leq e \sum_{n\leq N} e^{-\lambda_n^\sigma /\Lambda}\int_{\partial \Omega} u_{n,\sigma}^2 ds  \leq e \int_{\partial \Omega} K_\sigma\left(x,x;\frac 1{\Lambda}\right) ds .
\]
By \eqref{upper bd for Robin kernel},
\[
\int_{\partial \Omega} K_\sigma\left(x,x;\frac 1{\Lambda}\right) ds 
 \ll_\Omega \Lambda^{\dim \Omega/2} .
\]
Thus we find a uniform upper bound
\[
\sum_{\lambda_n^\sigma\leq \Lambda} \int_{\partial \Omega} u_{n,\sigma}^2 ds \ll_\Omega \Lambda^{\dim\Omega/2} \approx N
\]
on using Weyl's law, that is for all $\sigma>0$
\[
\frac 1N \sum_{n\leq N} \int_{\partial \Omega} u_{n,\sigma}^2 ds \leq C(\Omega) .
\]
\end{proof}

We note that the mean value result is valid in any dimension $d\geq 2$ for piecewise smooth domains $\Omega\subset\R^d$ as in \cite{HZ}, in the form  
\[
 \lim_{N\to \infty}  \frac 1N\sum_{n\leq N}   d_n(\sigma)  
 = \frac{2\Vol_{d-1}(\partial \Omega)}{\Vol_d (\Omega)}  \sigma .
\]
Indeed \cite{HZ}  prove the local Weyl law in that context, and Lemma~\ref{uniform local weyl law} is also valid in any dimension.

\section{A uniform lower bound for the gaps}
To obtain the lower bound of Theorem~\ref{thm:RN lower bound} for the gaps, we use the variational formula \eqref{formula for dn} to relate the derivative $d\lambda_n^\sigma/d\sigma$ to the boundary integrals $\int_{\partial \Omega}u_{n,\sigma}^2 ds$,
where $u_{n,\sigma}$ is any eigenfunction with eigenvalue $\lambda_n^\sigma$, and for that will require a lower bound on these boundary integrals.

\subsection{A lower bound for the boundary integral.}
The goal here is to prove a uniform  lower bound for the boundary data of Robin eigenfunctions on a star-shaped, 
smooth  planar domain $\Omega$.
\begin{thm}\label{thm:lower bound for tension}
Let $\Omega\subset \R^2$ be a  star-shaped  bounded planar domain with 
smooth boundary. Let $f $ be an $L^2(\Omega)$ normalized Robin eigenfunction associated with the $n$-th 
eigenvalue $\lambda_n^\sigma$.
Then   there are constants $C>0$, $A,B\geq 0$ depending on $\Omega$ so that for all $n\geq 1$, 
\begin{equation}\label{explicit lower bound on c for n=1}
 \int_{\partial \Omega}  f^2   ds \geq  \frac{1}{A\sigma^2 +B\sigma +C}>0.
\end{equation}
\end{thm}

For $\sigma=0$ (Neumann problem), this is related
to the $L^2$ restriction bound of Barnett-Hassell-Tacy \cite[Proposition 6.1]{BHT}.

\subsection{The Neumann case $\sigma=0$.}
We first show the corresponding statement for Neumann eigenfunctions (which are Robin case with $\sigma=0$), which is much simpler.
Let $f$ be a Neumann eigenfunction, that is $(\Delta+\lambda)f=0$ in $\Omega$, $\frac{\partial f}{\partial n}=0$ in $\partial \Omega$. We  may assume that $\lambda>0$, the result being obvious for $\lambda=0$ when $f$ is a constant function.  After translation, we may assume    that the domain is star-shaped with respect to the origin.

We start with a Rellich identity (\cite[Eq 2]{Rellich}): Assume that $\Omega\subset \R^d$ is a Lipschitz domain. Let $L=\Delta+\lambda$, and $A=\sum_{j=1}^d x_j\frac{\partial}{\partial x_j}$.
For every function $f$ on $\Omega$
\begin{multline}\label{Rellich eq 2}
\int_\Omega (Lf)(Af) dx  = \int_{\partial \Omega} \frac{\partial f}{\partial n}Af -\frac 12 \int_{\partial \Omega}||\nabla f||^2 \left(\sum_{j=1}^d x_j\frac{\partial x_j}{\partial n} \right)
\\
 + \frac{\lambda}{2} \int_{\partial \Omega} f^2 \left( \sum_{j=1}^d x_j\frac{\partial x_j}{\partial n}  \right)
-\frac{d}{2}\lambda \int_{\Omega} f^2 dx  + \left(\frac d2-1\right)\int_\Omega ||\nabla f||^2 dx .
\end{multline}

Using \eqref{Rellich eq 2} in dimension $d=2$ for a normalized eigenfunction, so that $Lf=0$ and $\int_\Omega f^2 =1$, and recalling that for Neumann eigenfunctions $ \frac{\partial f}{\partial n}=0$ on $\partial \Omega$, gives
\[
0=-\frac 12 \int_{\partial \Omega}||\nabla f||^2 \left(x\frac{\partial x}{\partial n}  + y\frac{\partial y}{\partial n} \right) +
\frac{\lambda}{2}   \int_{\partial \Omega} f^2 \left(x\frac{\partial x}{\partial n}  + y\frac{\partial y}{\partial n} \right)
-\lambda
\]
or
\[
\int_{\partial \Omega} \left(f^2-\frac 1{\lambda}||\nabla f||^2 \right)  \left(x\frac{\partial x}{\partial n}  + y\frac{\partial y}{\partial n} \right) ds
= 2 .
\]

 The term $x\frac{\partial x}{\partial n}  + y\frac{\partial y}{\partial n}$ is the inner product $n(\vec x) \cdot \vec x$ between the outward unit normal  $n(\vec x) = (\frac{\partial x}{\partial n} , \frac{\partial y}{\partial n})$ at the point $\vec x\in \partial \Omega$ and the radius vector $\vec x=(x,y)$ joining $\vec x$ and the origin.
Since the domain is star-shaped w.r.t. the origin, we have on the boundary $\partial \Omega$  
\[
x\frac{\partial x}{\partial n}  + y\frac{\partial y}{\partial n} = n(\vec x) \cdot \vec x \geq 0
\]
so that we can drop\footnote{If we also allow negative Robin constant $\sigma<0$, we may have a finite number of negative eigenvalues and this part of the argument would not work  for these.} the term with $||\nabla f||^2$ and get an inequality
\[
\int_{\partial \Omega}   (n(\vec x) \cdot \vec x )  f^2 ds \geq 2.
\]
Replacing $ (n(\vec x) \cdot \vec x ) \leq 2C_\Omega$  on $\partial \Omega$ gives Theorem~\ref{thm:lower bound for tension} for
$\sigma=0$:
\[
\int_{\partial \Omega}  f^2 \geq \frac{1}{C_\Omega} .
\]


\subsection{The Robin case}

Using the Rellich identity  \eqref{Rellich eq 2} in dimension $d=2$ for a normalized eigenfunction, so that $Lf=0$ and $\int_\Omega f^2 =1$, gives
\[
0=  \int_{\partial \Omega} \frac{\partial f}{\partial n}Af -\frac 12 \int_{\partial \Omega}||\nabla f||^2 (n(\vec x) \cdot \vec x )
+ \frac{\lambda}{2} \int_{\partial \Omega} f^2 (n(\vec x) \cdot \vec x )
-\lambda .
\]

 Now   $n(\vec x) \cdot \vec x\geq 0$ on the boundary $\partial \Omega$ since
$\Omega$ is star-shaped with respect to the origin, and $\lambda>0$, so we may drop the term with $||\nabla f||^2$ and get an inequality
\[
\int_{\partial \Omega}  f^2(\vec x)   (n(\vec x) \cdot \vec x) ds + \frac{2}{\lambda}\int_{\partial \Omega} \frac{\partial f}{\partial n}Af \geq 2 .
\]

Due to the boundary condition, we may replace the normal derivative $\frac{\partial f}{\partial n}$ by $-\sigma f$,
and  obtain, after using $0 \leq  n(\vec x) \cdot \vec x \leq 2C=2C_\Omega$ (we may take $2C$ to be the diameter of $\Omega$), that
\begin{equation}\label{rob rellich intermediate}
 C \int_{\partial \Omega}  f^2   - \frac{ \sigma }{\lambda}\int_{\partial \Omega} f (Af )\geq 1 .
\end{equation}
To proceed further, we need:
\begin{lem}
Assume that $\partial\Omega$ is smooth. There are  numbers $P,Q\geq 0$,  not both zero, depending only on $\partial \Omega$, so that for any normalized $\sigma$-Robin
eigenfunction $f$,
\begin{equation}\label{ineq for fAf}
\left| \int_{\partial \Omega} f (Af) ds \right| \leq  (P+\sigma Q) \int_{\partial \Omega} f^2ds  .
\end{equation}
\end{lem}
\begin{proof}
Decompose the vector field $A=x\frac{\partial }{\partial x} + y\frac{\partial }{\partial y}$ into its normal and tangential components along the boundary:
\[
A =p \frac{\partial}{\partial n} + q \frac{\partial}{\partial \tau}
\]
where $p,q$ are functions on the boundary $\Omega$. For example, for the circle $x^2+y^2=\rho^2$, we have $A=\rho \frac{\partial}{\partial n}$ and the normal derivative is just the radial derivative $ \frac{\partial}{\partial n}= \frac{\partial}{\partial r}$,
so that  $p \equiv \rho$, and $q  \equiv 0$.

Then using the Robin condition $\frac{\partial f}{\partial n} =-\sigma f$ on $\partial \Omega$ gives
\[
\int_{\partial \Omega} f (Af) ds  = \int_{\partial \Omega} f \left(p \frac{\partial f}{\partial n} +q \frac{\partial f}{\partial \tau}\right) ds
= -\sigma \int_{\partial \Omega}p f^2 ds+ \int_{\partial \Omega} q f  \frac{\partial f}{\partial \tau}  ds .
\]
Setting $P:=\max_{\partial \Omega}|p|$, we have
\[
 \left|-\sigma \int_{\partial \Omega}p f^2 ds\right|\leq \sigma P \int_{\partial \Omega} f^2ds
\]
so it remains to bound $\left| \int_{\partial \Omega} q f  \frac{\partial f}{\partial \tau}  ds\right|$.

Let $\gamma:[0,L]\to \partial \Omega$ be an arclength parameterization  with $\gamma(0) = \gamma(L)$.
Then note that the tangential derivative of $f$  at $x_0 = \gamma(s_0)$ is
\[
 \frac{\partial f}{\partial \tau}(x_0) = \frac{d}{ds}f(\gamma(s))\Big|_{s=s_0}
\]
and hence
\[
 f  \frac{\partial f}{\partial \tau} = \frac 12  \frac{\partial  (f^2) }{\partial \tau}=\frac 12  \frac{d  }{ds} \left\{f(\gamma(s))^2\right\}
\]
so that abbreviating $q(s) = q(\gamma(s))$ and integrating by parts
 \[
 \begin{split}
 \int_{\partial \Omega} q f  \frac{\partial f}{\partial \tau}  ds  &= \frac 12 \int_0^L  q(s) \frac{d  }{ds} \left\{f(\gamma(s))^2\right\} ds
 \\
 & =
 \frac 12 q(s)f(\gamma(s))^2\Big|_0^L - \frac 12 \int_0^L q'(s) f(\gamma(s))^2 ds .
 \end{split}
\]
Because the curve is closed: $\gamma(L)=\gamma(0)$, the boundary terms cancel out:
\[
q(s)f(\gamma(s))^2\Big|_0^L  = q(\gamma(L)) f(\gamma(L))^2-  q(\gamma(0)) f(\gamma(0))^2 =0
\]
and so
\[
\left|\int_{\partial \Omega} q f  \frac{\partial f}{\partial \tau}  ds\right|  = \left| \frac 12 \int_0^L q'(s) f(\gamma(s))^2 ds\right|
\leq Q \int_{\partial \Omega} f^2 ds
\]
where $Q=\max_{\partial \Omega}|\frac{dq}{d\tau}|$. Altogether we found that
\[
\left| \int_{\partial \Omega} f (Af) ds \right| \leq (\sigma P+Q) \int_{\partial \Omega} f^2 ds .
\]
\end{proof}
%
%

We may now conclude the proof of Theorem~\ref{thm:lower bound for tension} for $\sigma>0$:
Take $f=u_{n,\sigma}$ the $n$-th eigenfunction, with $n\geq 1$. 
Inserting \eqref{ineq for fAf} into \eqref{rob rellich intermediate}  we find 
\[
1\leq C \int_{\partial \Omega}  f^2   - \frac{ \sigma }{\lambda}\int_{\partial \Omega} f (Af )
\leq \left(C+\frac{ \sigma(P+Q\sigma)}{\lambda_n^\sigma} \right) \int_{\partial \Omega} f^2 ds .
\]
Hence we find, on replacing $\lambda_n^{\sigma}\geq \lambda_1^{\sigma} \geq \lambda_1^0>0$, that 
\begin{equation*}
 \int_{\partial \Omega}  f^2   ds \geq  \frac{1}{C+ \sigma(P+Q\sigma)/ \lambda_1^0}>0 
\end{equation*}
which is of the desired form. 
\qed

\subsection{Proof of Theorem~\ref{thm:RN lower bound}}

We use the variational formula \eqref{formula for dn}    for $n\geq 1$ with the lower bound \eqref{explicit lower bound on c for n=1} of Theorem~\ref{thm:lower bound for tension}
\[
d_n(\sigma) = \int_0^\sigma \left(\oint_{\partial\Omega} u_{n,\tau}^2 ds \right)d\tau 
\geq   \int_0^\sigma \frac{d\tau}{A\tau^2+B\tau+C  } =:c_1(\Omega,\sigma) >0.
\]
 
For $n=0$, we just use  positivity  of the RN gap $d_0(\sigma)>0$, and finally deduce that for all $n\geq 0$, and $\sigma>0$, 
\[
d_n(\sigma)\geq c(\Omega,\sigma):=\min\left(c_1(\Omega,\sigma), d_0(\sigma) \right)>0.
\]
\qed

\section{Ergodic billiards}

In this section we give a proof of Theorem~\ref{thm:QE RN gaps}.
By Chebyshev's inequality,  it suffices to show:
\begin{prop}
Let $\Omega\subset \R^2$ be a bounded, piecewise smooth domain. Assume that the billiard map for $\Omega$ is ergodic. Then for every $\sigma>0$,
\begin{equation}\label{markov}
\lim_{N\to \infty} \frac 1N\sum_{n\leq N} \left|d_n(\sigma)-\frac{2\length(\partial \Omega)}{\area (\Omega)} \cdot \sigma \right| = 0 .
\end{equation}
 \end{prop}

\begin{proof}

We again use the variational formula \eqref{formula for dn}
\[
d_n(\sigma)  = \int_0^\sigma \left( \int_{\partial \Omega} u_{n,\tau}^2 ds  \right) d\tau  .
\]
 We have
\[
\begin{split}
 \left| d_n(\sigma)- \frac{2\length(\partial \Omega)}{\area (\Omega)}\sigma   \right|
&=
\left|   \int_0^\sigma \Big( \int_{\partial \Omega} u_{n,\tau}^2 ds  \Big) d\tau   - \frac{2\length(\partial \Omega)}{\area (\Omega)} \sigma \right|
\\
&=  \left| \int_0^\sigma \Big(\int_{\partial \Omega} u_{n,\tau}^2 ds - \frac{2\length(\partial \Omega)}{\area (\Omega)}  \Big) d\tau  \right|
\\
&\leq  \int_0^\sigma \left|\int_{\partial \Omega} u_{n,\tau}^2 ds   - \frac{2\length(\partial \Omega)}{\area (\Omega)} \right| d\tau .
\end{split}
\]
 Therefore
 \[
 \begin{split}
 \frac 1N\sum_{n\leq N} \left| d_n(\sigma)- \frac{2\length(\partial \Omega)}{\area (\Omega)}\sigma \right|  & \leq
 \int_0^\sigma   \frac 1N\sum_{n\leq N}   \left|\int_{\partial \Omega} u_{n,\tau}^2 ds   - \frac{2\length(\partial \Omega)}{\area (\Omega)} \right|   d\tau
 \\
 &=: \int_0^\sigma S_N(\tau) d\tau
\end{split}
\]
where
\[
S_N(\tau):= \frac 1N\sum_{n\leq N}   \left|\int_{\partial \Omega} u_{n,\tau}^2 ds   - \frac{2\length(\partial \Omega)}{\area (\Omega)} \right| .
\]

Hassell and Zelditch \cite[eq 7.1]{HZ} (see also  Burq \cite{Burq}) show that if the billiard map is ergodic then for each $\sigma\geq 0$,
\begin{equation}\label{HZ variance}
\lim_{N\to \infty}
\frac 1N\sum_{n\leq N} \left| \int_{\partial\Omega} u_{n,\sigma}^2 ds-\frac{2\length(\partial \Omega)}{\area (\Omega)} \right|^2 = 0 .
\end{equation}
Therefore, by Cauchy-Schwarz, $S_N(\tau)$ tends to zero for all $\tau\geq 0$, by \eqref{HZ variance}; by Lemma~\ref{uniform local weyl law}  we know that $S_N(\tau)\leq C$ is uniformly bounded for all $\tau\leq \sigma$,
so that by the Dominated Convergence Theorem we  deduce that the limit of the integrals tends to zero, hence that
\[
\lim_{N\to \infty}  \frac 1N\sum_{n\leq N} \left| d_n(\sigma)- \frac{2\length(\partial \Omega)}{\area (\Omega)}\sigma \right|    =  0.
\]
 \end{proof}

We note that  Theorem~\ref{thm:QE RN gaps} is valid in any dimension $d\geq 2$ for piecewise smooth domains $\Omega\subset\R^d$ with ergodic billiard map as in \cite{HZ}, with the mean value   interpreted as $\frac{2\Vol_{d-1}(\partial \Omega)}{\Vol_d (\Omega)}  \sigma$. 

\section{ Variable Robin function} \label{sec:Variable sigma}
In this section, we indicate extensions of our general results to the case of variable boundary conditions. 
\subsection{Variable boundary conditions} 
The general Robin boundary condition is obtained by taking    a function on the boundary  $\sigma: \partial \Omega\to \R$ which we assume is always non-negative: $\sigma(x)\geq 0$ for all $x\in \partial \Omega$. Thus we look for solutions of 
\[
\Delta u +\lambda u=0 \;{\rm on} \;\Omega,
\]
\[
 \frac{\partial u}{\partial n}(x) +\sigma(x) u(x) = 0, \quad x\in \partial \Omega 
\]
which is interpreted in weak form as saying that 
\begin{equation*}
\int_\Omega \nabla u_{n } \cdot \nabla v +\oint_{\partial \Omega}  \sigma  u_{n} v    = \lambda_n  \int_\Omega u_{n } v
\end{equation*}
for all $v\in H^1(\Omega)$. 
We will assume that $\sigma$ is continuous. Then  we obtain positive Robin eigenvalues
\[
0<\lambda_0^\sigma\leq \lambda_1^\sigma\leq \dots
\]
except that in the Neumann case $\sigma\equiv 0$ we also have zero as an eigenvalue.

Robin to Neumann bracketing is still valid here, in the following form: if $\sigma_1,\sigma_2\in C(\partial \Omega)$ are two continuous functions with $0\leq \sigma_1\leq \sigma_2$ and such that there is some point $x_0\in \partial \Omega$ such that there is strict inequality $\sigma_1(x_0)<\sigma_2(x_0)$ (by continuity this therefore holds on a neighborhood of $x_0$), then we have a strict inequality \cite{Rohleder} 
\begin{equation}\label{RNbracket} 
\lambda_n^{\sigma_1}<\lambda_n^{\sigma_2}, \quad \forall n\geq 1.
\end{equation}

Fix such a Robin function $\sigma\in C(\partial \Omega)$, which is positive $\sigma> 0$. We are interested in the Robin-Neumann gaps 
\[
d_n(\sigma):=\lambda_n^\sigma - \lambda_n^0 
\]
which are positive by \eqref{RNbracket}. 

\subsection {Extension of general results}
The  lower and upper bounds of Theorems~\ref{thm:RN lower bound} and \ref{general smooth upper bd}  remain valid for variable $\sigma$ by an easy reduction to the constant case:  Let
\[
\sigma_{\min}=\min_{x\in \partial \Omega} \sigma(x), \quad \sigma_{\max}=\max_{x\in \partial \Omega} \sigma(x)
\]
so that $0<\sigma_{\min}\leq \sigma_{\max}$ (with equality only if $\sigma$ is constant). Using \eqref{RNbracket} gives
\[
\lambda_n^0<\lambda_n^{\sigma_{\min}}\leq \lambda_n^{\sigma}\leq \lambda_n^{\sigma_{\max}}
\]
so that 
\[
d_n(\sigma_{\min})\leq d_n(\sigma)\leq d_n(\sigma_{\max}).
\]

For instance, the universal lower bound for star-shaped domains (Theorem~\ref{thm:RN lower bound}) follows because $d_n(\sigma)\geq d_n(\sigma_{\min}) \geq C(\sigma_{\min})>0$, etcetera.

 The existence of mean values (Theorem~\ref{thm:mean of E RN gaps}) and the almost sure convergence of the gaps to the mean value in the ergodic case (Theorem~\ref{thm:QE RN gaps})  require an adjustment of the variational formula (Lemma~\ref{lem:formula for dn}) which is provided in \S~\ref{sec:var lem}. 
 Once that is in place, 
%
the result is  
\begin{equation}\label{mean value var}
\lim_{N\to \infty} \frac 1N \sum_{n=1}^N d_n(\sigma) = \frac{2\oint_{\partial \Omega}\sigma(s)ds}{\area (\Omega)} .
\end{equation}


Given the mean value formula \eqref{mean value var},  Theorem~\ref{thm:QE RN gaps} (almost sure convergence of the RN gaps to the mean in the ergodic case) also follows.

\subsection{A variational formula}\label{sec:var lem}

 Let  $\Omega\subset \R^d$ be a bounded Lipschitz domain.   
Fix a continuous, positive Robin function $\sigma:\partial \Omega\to \R_{>0}$. 
We consider a one-parameter deformation of the boundary value problem $\Delta u+\lambda u=0$, 
\begin{equation}\label{deformed robin}
 \frac{\partial u}{\partial n}(x) +\alpha \sigma(x) u(x) = 0, \quad x\in \partial \Omega  
\end{equation}
with a real parameter $\alpha\geq 0$. Denote the corresponding eigenvalues by 
\[
\lambda_1(\alpha)\leq \lambda_2(\alpha)\leq \dots \leq \lambda_n(\alpha)\leq \dots
\]
By Robin-Neumann bracketing, if $0\leq \alpha_1<\alpha_2$ then
\[
\lambda_n(\alpha_1)<\lambda_n(\alpha_2), \quad \forall n\geq 1 .
\]
The previous RN gaps $d_n(\sigma)$ are precisely $\lambda_n(1)-\lambda_n(0)$.  The variational formula for the RN gaps is:

\begin{lem}
 Let  $\Omega\subset \R^d$ be a bounded Lipschitz domain. Then
\begin{equation*}
d_n(\sigma)  =  \int_0^1 \left( \int_{\partial \Omega} |u_{n,\alpha}|^2 ds  \right) d\alpha
\end{equation*}
where $u_{n,\alpha}$ is any $L^2(\Omega)$-normalized eigenfunction associated with $\lambda_n(\alpha)$.
\end{lem}

The  proof is identical to that of Lemma~\ref{lem:formula for dn}, except that we need a reformulation of   \cite[Lemma 2.11]{AFK} to this context\footnote{\cite[Lemma 2.11]{AFK} allows $\Omega \subset \R^N$ to be any bounded Lipschitz domain and takes $\sigma\equiv 1$. }:  
\begin{lem}
 Let  $\Omega\subset \R^d$ be a bounded Lipschitz domain and $\sigma>0$ a positive continuous  function on the boundary $\partial \Omega$.  For $\alpha\geq 0$, let $\lambda_n(\alpha)$ be the eigenvalues of the Robin eigenvalue problem \eqref{deformed robin}

Then for $n \geq 1$, $\lambda_n(\alpha)$ 
is an absolutely continuous and strictly increasing function of $\alpha\in [0,\infty)$, which is differentiable almost everywhere in  $(0,\infty)$.  
Where it exists, its derivative is given by
\begin{equation}\label{new var formula b}
\frac{d}{d\alpha} \lambda_n(\alpha) = \frac{ \oint_{\partial \Omega} \sigma  u_{n,\alpha}^2 }{\int_\Omega  u_{n,\alpha}^2}
\end{equation}
where $ u_{n,\alpha}\in H^1(\Omega)$  is any eigenfunction associated with $\lambda_n(\alpha)$. 

 \end{lem}
 \begin{proof} 
 The proof is verbatim that of \cite[Lemma 2.11]{AFK} where $\sigma\equiv 1$. As is explained there, each eigenvalue depends locally analytically on $\alpha$, with at most a locally finite set of splitting points. We just repeat the computation of the derivative at any 
  $\alpha$ which is not a splitting point for $\lambda_n(\alpha)$:
 We use the weak formulation of the boundary condition, as saying that for all $v\in H^1(\Omega)$, 
\begin{equation}\label{weak form t}
\int_\Omega \nabla  u_{n,\alpha} \cdot \nabla v +\oint_{\partial \Omega} \alpha\sigma(s)  u_{n,\alpha}(s) v(s) ds = \lambda_n(\alpha) \int_\Omega  u_{n,\alpha} v .
\end{equation}
In particular, applying \eqref{weak form t} with   $v=u_{n,\beta}$ gives
\begin{equation}\label{weak form t1}
\int_\Omega \nabla u_{n,\alpha} \cdot \nabla u_{n,\beta} +\oint_{\partial \Omega} \alpha\sigma  u_{n,\alpha}  u_{n,\beta} ds = \lambda_n(\alpha) \int_\Omega u_{n,\alpha}  u_{n,\beta}.
\end{equation}
Changing the roles of $\alpha$ and $\beta$ gives
\begin{equation}\label{weak form t2}
\int_\Omega \nabla u_{n,\alpha} \cdot \nabla u_{n,\beta} +\oint_{\partial \Omega} \beta\sigma  u_{n,\alpha} u_{n,\beta}  ds = \lambda_n(\beta) \int_\Omega u_{n,\alpha}  u_{n,\beta}.
\end{equation}
Subtracting \eqref{weak form t1} from \eqref{weak form t2} gives
\begin{equation*}
\frac{\lambda_n(\beta)-\lambda_n(\alpha)}{\beta-\alpha} = \frac{\oint_{\partial \Omega}  \sigma(s) u_{n,\alpha}(s) u_{n,\beta}(s) ds }{ \int_\Omega u_{n,\alpha}  u_{n,\beta}} .
\end{equation*}
Taking the limit $  \beta\to \alpha$ and assuming that $u_{n,\beta}\to u_{n,\alpha} $ in $H^1(\Omega)$ as $\beta \to \alpha$, as verified in  \cite[Lemma 2.11]{AFK} so that in particular the denominator is eventually nonzero, gives \eqref{new var formula b}. 
\end{proof}

 \section{Boundedness of RN gaps for rectangles} \label{sec:rectangles}
We consider the rectangle $Q_{L}=[0,1]\times [0,L]$, with $L\in (0,1]$ the aspect ratio. We denote by $\lambda_0^\sigma\leq \lambda_1^\sigma\leq \dots$ the ordered Robin eigenvalues. We will prove Theorem~\ref{thm:uniform upper bound}, that
\[
0<\lambda_n^\sigma-\lambda_n^0  <C_L(\sigma)   .
\]
\subsection{The one-dimensional case}
Let $ \sigma>0$ be the Robin constant.
The Robin problem on the unit interval is $-u_n''=k_n^2 u_n$, with the one-dimensional Robin boundary conditions
\[
-u'(0) +\sigma u(0)=0, \quad u'(1)+\sigma u(1)=0  .
\]
The eigenvalues of the Laplacian on the unit interval are the numbers $-k_n^2$ where the frequencies $k_n=k_n(\sigma)$ are the solutions of the secular equation $(k^2-\sigma^2)\sin k = 2k\sigma \cos k$, or
\begin{equation}\label{secular eq}
\tan (k) =\frac{2\sigma k}{k^2-\sigma ^2}
\end{equation}
(see Figure~\ref{fig:seceqplot}) and the corresponding eigenfunctions are 
\[
u_n(x)=k_n\cos(k_nx)+\sigma \sin(k_n x).
\] 


 \begin{figure}[ht]
\begin{center}
\includegraphics{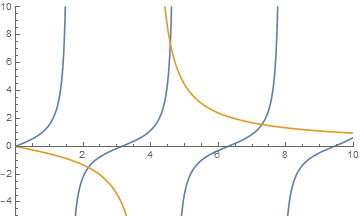}
\caption{ The secular equation \eqref{secular eq} for $\sigma=4$. Displayed are plots of $\tan k$ vs. $\frac{2\sigma k}{k^2-\sigma ^2}$.}
\label{fig:seceqplot}
\end{center}
\end{figure}

As a special case\footnote{Of course, in this case it directly follows from the secular equation \eqref{secular eq}.}
of Dirichlet-Neumann bracketing \eqref{NRD bracketing}, we know that given $\sigma>0$, for each $n\geq 0$
there is a unique solution $k_n=k_n(\sigma)$ of the secular equation \eqref{secular eq}
with
$$
k_n\in ( n\pi, (n+1)\pi), \quad n\geq 0.
$$
Note that $k_n(0) = n\pi$.


 From \eqref{secular eq}, we have as $n\to \infty$,
\[
k_n(\sigma)=n\pi +\arctan\left( \frac{2\sigma}{k_n(\sigma)} \frac 1{1-\frac{\sigma^2} {k_n(\sigma)^2}} \right) =n\pi +\frac{2\sigma}{k_n(\sigma)} +O\left( k_n(\sigma)^{-3}\right)
\]
so that 
\begin{equation}\label{asymp one dim}
k_n(\sigma)^2-k_n(0)^2 \sim 4\sigma, \quad n\to \infty .
\end{equation}
We can interpret, for $\Omega$ being the unit interval, $4=2\#\partial \Omega/\length \Omega$ so that we find convergence of the RN gaps to their mean value in this case.

From \eqref{asymp one dim} we deduce: 

\begin{lem}
For every $\sigma>0$, there is some $C(\sigma)>0$ so that 
\begin{equation}\label{eq:one dim gap}
k_n(\sigma)^2-k_n(0)^2 \leq C(\sigma), \quad \forall n\geq 0.
\end{equation}
\end{lem}

\subsection{Proof of Theorem~\ref{thm:uniform upper bound}}\label{sec:pass to ordered evs}
The frequencies for the interval $[0,L]$ are $\frac{1}{L}\cdot k_{m}(\sigma \cdot L)$.
 Hence the  Robin energy levels of $Q_{L}$ are the numbers
\begin{equation} \label{Lambda for square}
\eigen_{n,m}(\sigma)=k_{n}(\sigma)^{2}+\frac{1}{L^{2}}\cdot k_{m}(\sigma \cdot L)^{2}, \quad n,m\geq 0.
\end{equation}
We have
\[
0\leq \eigen_{n,m}(\sigma)-\eigen_{n,m}(0)
=(k_{n}(\sigma)^{2}-k_{n}(0)^{2}) +  \frac{1}{L^{2}}\cdot  \left(k_{m}(\sigma \cdot L)^{2}-k_{m}(0)^2\right) .
\]
From the one-dimensional result \eqref{eq:one dim gap}, we   deduce that
\[
\eigen_{n,m}(\sigma)-\eigen_{n,m}(0)  \leq C(\sigma) + \frac 1{L^2}C(L\sigma) =    C_L(\sigma).
\]

 We now pass from the $\eigen_{m,n}(\sigma)$ to the ordered eigenvalues $\{\lambda_k^\sigma: k=0,1,\dots\}$.
 We know that $\lambda_k^\sigma\geq \lambda_k^0$, and want to show that $\lambda_k^\sigma\leq \lambda_k^0+C_L(\sigma)$.
 For this it suffices to show that the interval $I_k:=[0,\lambda_k^0+C_L(\sigma)]$ contains at least $k+1$ Robin eigenvalues, since then it will contain $\lambda_0^\sigma,\dots, \lambda_k^\sigma$ and hence we will find $\lambda_k^\sigma\leq \lambda_k^0+C_L(\sigma)$.

 The interval  $I_k$ contains the interval $[0,\lambda_k^0]$ and so certainly contains the first $k+1$ Neumann eigenvalues $\lambda_0^0,\dots,\lambda_k^0$, which are of the form $\Lambda_{m,n}(0)$ with $(m,n)$ lying in a set $\mathcal S_k$. Since $\Lambda_{m,n}(\sigma)\leq \Lambda_{m,n}(0)+C_L(\sigma)$, the interval $I_k$ must contain the $k+1$ eigenvalues  $\{\Lambda_{m,n}(\sigma)
:(m,n)\in \mathcal S_k\}$, and we are done.
 \qed

\section{Application of boundedness of the RN gaps to level spacings}\label{sec:level spacing}

In this section, we show that the level spacing distribution of the Robin eigenvalues for the desymmetrized square is a delta function at the origin, as is the case with Neumann or Dirichlet boundary conditions.

Recall the definition of the level spacing distribution: We are given  a sequence of levels $x_0\leq x_1\leq x_2\leq \dots$. We assume that $x_N=cN+o(N)$, as is the case of the eigenvalues of a planar domain. Let
$\delta_n=(x_{n+1}-x_n)/c$ be the normalized nearest neighbour gaps. so that the average gap is unity. The level spacing distribution $P(s)$ of the sequence is then defined as
\[
\int_0^y P(s)ds = \lim_{N\to \infty} \frac 1N \#\{n\leq N: \delta_n\leq y\}
\]
(assuming that the limit exists).

Recall that the Robin spectrum has systematic double multiplicities $\Lambda_{m,n}(\sigma) = \Lambda_{n,m}(\sigma)$ (see \eqref{Lambda for square} with $L=1$), which forces half the gaps to vanish for a trivial reason. To avoid this issue, one takes only the levels $\Lambda_{m,n}(\sigma)$ with $m\leq n$, which we call the desymmetrized Robin spectrum.

\begin{thm}
\label{thm:level spacing delta}
For every $\sigma\geq 0$, the level spacing distribution for the desymmetrized
Robin spectrum on the square is a delta-function at the origin.
\end{thm}
In other words, if we denote by $\lambda_0^\sigma\leq \lambda_1^\sigma\leq \dots$ the ordered (desymmetrized) Robin eigenvalues, then the cumulant of the level spacing distribution satisfies: For all $y>0$,
\[
\int_0^y P(s)ds = \lim_{N\to \infty} \frac 1N \#\left\{n\leq N:
\frac 12\frac{\area(\Omega)}{4\pi}(\lambda_{n+1}^\sigma-\lambda_n^\sigma) \leq y\right\} = 1.
\]

\begin{proof}
The Neumann spectrum  for the square consists of the numbers $ m^2+n^2$ (up to a multiple), with $m,n\geq 0$. There is a systematic double multiplicity, manifested by the symmetry $(m,n)\mapsto (n,m)$. We remove it by requiring $m\leq n$.
Denote the integers which are sums of two squares by
\[
s_1=0<s_2=1<s_3=2<s_4=4 <s_5=5<\dots<s_{14}=25<\dots
\]
We define index clusters $\mathcal N_i$ as the set of all indices of desymmetrized Neumann eigenvalues which coincide with $s_i$:
\[
\mathcal N_i = \{n: \lambda_n^0=s_i\}
\]
For instance, $s_0=0=0^2+0^2$ has multiplicity one, and gives the index set $\mathcal N_1=\{1\}$; $s_1=1= 0^2+1^2$ has multiplicity $1$  (after desymmetrization)   and gives $\mathcal N_2=\{2\}$;   $s_3=2=1^2+1^2$ giving $\mathcal N_3 = \{3\}$, $\dots$
$s_{14}=25=   0^2+5^2=3^2+4^2$, $\mathcal N_{14} = \{14,15 \}$, etcetera.
Then these are sets of consecutive integers which form a partition of the natural numbers $\{1,2,3,\dots\}$, and if $i<j$ then the largest integer in $\mathcal N_i$ is smaller than the smallest integer in $\mathcal N_j$.

Denote by $\lambda_n^\sigma$ the ordered desymmetrized Robin eigenvalues: $\lambda_0^\sigma\leq \lambda_1^\sigma\leq \dots$,
so for $\sigma=0$ these are just the integers $s_i$ repeated with multiplicity $\#\mathcal N_i$.
For each $\sigma\geq 0$, we define clusters $C_i(\sigma)$ as the set of all  desymmetrized Robin eigenvalues $\lambda_n^\sigma$ with $n\in \mathcal N_i$:
\[
C_i(\sigma) = \{\lambda_n^\sigma: n\in \mathcal N_i\} .
\]

Now use the boundedness of the RN gaps (Theorem~\ref{thm:uniform upper bound}):
$0\leq \lambda_n^\sigma-\lambda_n^0\leq C(\sigma)$,
to deduce that  the clusters have bounded diameter:
\[
\diam C_i(\sigma)\leq C(\sigma) .
\]
If  $\#\mathcal N_i=1$ then $\diam C_i(\sigma)=0$, so we may assume that $\#\mathcal N_i\geq 2$ and write
\[
\mathcal N_i = \{n_-,n_-+1,\dots, n_+\}, \quad n_+ =\max \mathcal N_i, \quad n_- =\min \mathcal N_i.
\]
  Then
\[
\begin{split}
\diam C_i(\sigma) &= \lambda_{n_+}^\sigma-\lambda_{n_-}^\sigma
\\
&= (\lambda_{n_+}^\sigma -s_i)+(s_i-\lambda_{n_-}^\sigma)
\\
& =
(\lambda_{n_+}^\sigma-\lambda_{n_+}^0)-(\lambda_{n_-}^\sigma-\lambda_{n_-}^0)\leq C(\sigma)-0 = C(\sigma) .
\end{split}
\]

For the first $N$ eigenvalues, the number $I$ of clusters containing them is the number of the $s_i$ involved, which is at most the number of $s_i\leq \lambda_N^\sigma  \approx N$.
 A classical result of Landau \cite{Landau} states that the number of integers $\leq N$ which are sums of two squares is about $N/\sqrt{\log N}$, in particular\footnote{This is much easier to show using a sieve.} is $o(N)$.
Hence
$$I\leq \#\{i: s_i \ll N\} =o(N) .
$$

We count the number of nearest neighbour\footnote{For simplicity we replace  $\frac 12 \frac{\area(\Omega)}{4\pi}$ by $1$, that is we don't bother normalizing so as to have mean gap unity;  the result is independent of this normalization.} gaps $\delta_n^\sigma=\lambda_{n+1}^\sigma-\lambda_n^\sigma$ of size bigger than $y$.
Of these, there are at most $I$ such that $\lambda_{n+1}^\sigma$ and $\lambda_n^\sigma$ belong to different clusters, and since $I=o(N)$ their contribution is negligible. For the remaining ones, we group them by cluster to which they belong:
\[
\# \{n\leq N: \delta_n^\sigma>y\} = \sum_{i=1}^I\#\left\{n: \lambda_{n+1}^\sigma, \lambda_n^\sigma\in C_i(\sigma) \; \& \; \delta_n^\sigma>y\right\}
+ o(N) .
\]
We have
\begin{multline*}
\#\left\{n: \lambda_{n+1}^\sigma, \lambda_n^\sigma\in C_i(\sigma)\; \& \; \delta_n^\sigma>y\right\} =
\#\left\{n\in\mathcal N_i , \;  n<\max \mathcal N_i , \; \delta_n^\sigma>y\right\}
\\
=
\sum_{\substack{n\in\mathcal N_i\\  n<\max \mathcal N_i \\ \delta_n^\sigma>y}}  \frac yy
 < \sum_{\substack{n\in\mathcal N_i\\   n<\max \mathcal N_i \\  \delta_n^\sigma>y}}  \frac {\delta_n^\sigma}y
\leq \frac 1y \sum_{\substack{n\in\mathcal N_i \\ n<\max \mathcal N_i }} \delta_n^\sigma  .
\end{multline*}

The sum of nearest neighbour gaps in each cluster is
\[
 \sum_{\substack{n\in\mathcal N_i \\ n<\max \mathcal N_i }} \delta_n^\sigma  =\sum_{\substack{n\in\mathcal N_i \\ n<\max \mathcal N_i }}( \lambda_{n+1}^\sigma-\lambda_n^\sigma)
  =\lambda_{\max \mathcal N_i}^\sigma-\lambda_{\min \mathcal N_i}^\sigma =\diam C_i(\sigma) \leq C(\sigma) .
\]
 Thus we find
\[
\#\{n: \lambda_{n+1}^\sigma, \lambda_n^\sigma\in C_i(\sigma)\; \& \; \delta_n^\sigma>y\} \leq \frac{C(\sigma)}{y}
\]
so that
\[
\# \{n\leq N: \delta_n^\sigma>y\} \leq \sum_{i=1}^I \frac{C(\sigma)}{y} + o(N)  = \frac{C(\sigma)}{y}I + o(N) .
\]
Since $I=o(N)$, and $C$, $y$ are fixed, we conclude that
\[
\frac 1N \# \{n\leq N: \delta_n^\sigma>y\}  = o(1) .
\]

Thus the cumulant of the level spacing distribution satisfies: For all $y>0$,
\[
\int_0^y P(s)ds = \lim_{N\to \infty} \frac 1N \#\{n\leq N: \delta_n^\sigma\leq y\} = 1
\]
so that $P(s) $ is a delta function at the origin.
\end{proof}

Note that the claim is not that all gaps $\lambda_{n+1}^\sigma-\lambda_n^\sigma$ tend to zero.   
On the contrary, it is possible to produce thin sequences $\{n\}$ so that $\lambda_{n+1}^\sigma-\lambda_n^\sigma$
tend to infinity. 
Looking at  the proof of Theorem \ref{thm:level spacing delta}, 
these correspond to the rare cases when $\lambda_{n}^{\sigma}$ and $\lambda_{n+1}^{\sigma}$ belong to neighboring ``clusters" which are far apart from each other. 

 \section{The unit disk}

 \subsection{Upper bounds for $d_n$ via Weyl's law}
 In this section we prove Theorem~\ref{thm:upper bd for disk}.
 We first show how to obtain upper bounds for the gaps $d_n$  from upper bounds in Weyl's law for the Robin/Neumann problem.
 The result is that
\begin{lem}\label{lem:dn and weyl}
Let $\Omega$ be a bounded planar domain. Assume that there is some $\theta\in (0,1/2)$ so that
\begin{equation}\label{Weyl law theta}
N_\sigma(x):=\#\{ \lambda_n^\sigma\leq x\} = \frac{\area(\Omega)}{4\pi}x +  \frac{\length(\partial\Omega)}{4\pi}\sqrt{x} + O_\sigma(x^{\theta}) .
\end{equation}
and the same result holds for $\sigma=0$. Then we have
\[
d_n(\sigma)\ll_\sigma n^{\theta} .
\]
\end{lem}
\begin{proof}
We first note that  \eqref{Weyl law theta} gives  
\begin{equation} \label{multiplicity bound}
N_0(\lambda_n^0) =n +O(n^\theta) 
\end{equation}
and likewise for the Robin counting function. 

Indeed, denote $\lambda=\lambda_n^0$, and pick $\varepsilon\in (0,1)$ sufficiently small so that  in the interval $[\lambda-\frac \varepsilon 2,  \lambda+\frac \varepsilon 2]$ there are no eigenvalues other than $\lambda$, which is repeated with multiplicity $K\geq 1$. Then $N(\lambda+\frac \varepsilon 2)-N(\lambda-\frac \varepsilon 2)=K$. On the other hand, by Weyl's law (with $A=\area(\Omega)/4\pi$, $B=\length(\partial\Omega)/4\pi$)  
\[
\begin{split}
K&=N(\lambda+\frac \varepsilon 2)-N(\lambda-\frac \varepsilon 2) 
\\
& = A(\lambda+\frac \varepsilon 2) +B\sqrt{\lambda+\frac \varepsilon 2} +O\left(\left(\lambda+\frac \varepsilon 2\right)^\theta\right) 
\\
&\quad  -  \left( A(\lambda-\frac \varepsilon 2) +B\sqrt{\lambda-\frac \varepsilon 2} +O\left( \left(\lambda-\frac \varepsilon 2\right)^\theta\right) \right) 
\\
&=A\varepsilon +O(\frac{\varepsilon}{\sqrt{\lambda}}) +O(\lambda^\theta).
\end{split}
\] 
Now use $|N(\lambda_n^0) - n|\leq K \ll \lambda^\theta \ll n^\theta $ which gives \eqref{multiplicity bound}.  

Now compare the counting functions $N_\sigma(\lambda_n^\sigma)$ and $N_0(\lambda_n^0)$ for the Robin and Neumann spectrum using \eqref{Weyl law theta} and \eqref{multiplicity bound}:
\[
n +O(n^\theta)= N_\sigma(\lambda_n^\sigma) = \frac{\area(\Omega)}{4\pi} \lambda_n^\sigma  +  \frac{\length(\partial\Omega)}{4\pi}\sqrt{\lambda_n^\sigma } + O_\sigma(n^{\theta}) .
\]
and 
\[
n+O(n^\theta)= N_0(\lambda_n^0) = \frac{\area(\Omega)}{4\pi} \lambda_n^0 +  \frac{\length(\partial\Omega)}{4\pi}\sqrt{\lambda_n^0 } + O(n^{\theta}) .
\]
Subtracting the two gives
\[
\Big( \lambda_n^\sigma - \lambda_n^0  \Big) \cdot \left(   \area(\Omega)+ \frac{  \length(\partial\Omega)}{ \sqrt{\lambda_n^\sigma}+\sqrt{\lambda_n^0}}\right)  = O_\sigma(n^{\theta})
\]
and therefore
\[
d_n(\sigma) = \lambda_n^\sigma - \lambda_n^0   = O_\sigma(n^{\theta})
\]
\end{proof}

Below we implement this strategy for the disk to obtain Theorem~\ref{thm:upper bd for disk}.

 \subsection{Relating Weyl's law and a lattice point count}\label{sec:weyllatdisk}

Define the domain
\[
D=\left\{ \left(x,y\right):\,x\in\left[-1,1\right],\max\left(0,-x\right)\le y\le g\left(x\right)\right\}
\]
where
\begin{equation}
\label{eq:g_def}
g\left(x\right)=\frac{1}{\pi}\left(\sqrt{1-x^{2}}-x\arccos x\right).
\end{equation}
\begin{figure}[ht]
\begin{center}
\includegraphics[height=40mm]{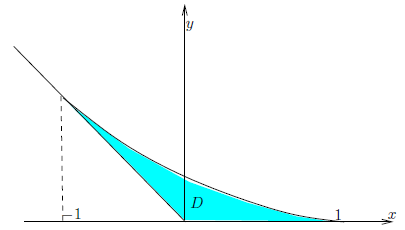}
\caption{The domain $D$.}
\label{thedomainD.fig}
\end{center}
\end{figure}
 Let
\[
N_{D}\left(\mu\right):=\#\left\{ \left(n,k\right):\left(n,k+\max(0,-n)-\frac 34 \right)\in\mu D\right\}
\]
and
\[
N_{\rm disk, \sigma}\left(x \right):=\#\left\{\lambda_n^\sigma\le x \right\} .
\]
 \begin{prop}\label{rel between weyl and lattice}
 Fix $\sigma\geq 0$. Then
 \[
N_{D}\left(\mu-\frac{C}{\mu^{3/7}}\right)-C\mu^{4/7}\le N_{\rm disk,\sigma}\left(\mu^2 \right)\le N_{D}\left(\mu+\frac{C}{\mu^{3/7}}\right)+C\mu^{4/7} .
\]
 \end{prop}
 The argument extends   \cite[Theorem 3.1]{CdV disk}, \cite{GMWW} (who fix a flaw in the argument of \cite{CdV disk})  to Robin  boundary conditions. 


We can now prove Theorem~\ref{thm:upper bd for disk}.
We use the result of \cite{GWW}
\footnote{They treat the shifted lattice $\Z^2-(0,\frac 14)$ but as they say \cite[Remark 6.5]{GWW},
the arguments also work for the shift by $(0,\frac 34)$. See \cite{GMWW}  for  an a further improvement in the Dirichlet case to
 $131/416=1/3-23/1248= 0.314904$. }
\[
N_{D}\left(\mu\right)  = \area(D)\mu^2+ \frac \mu 2+O\Big(\mu^{2(1/3-\delta)}\Big)
\]
where $ \delta=1/990$.
Noting that
\[
\area(D) = \frac{\area (\Omega)}{4\pi}=\frac 14,\quad \frac{ \length(\partial \Omega)}{4\pi} = \frac 12
\]
we obtain from Proposition~\ref{rel between weyl and lattice} that
\[
N_{\rm disk,\sigma}\left(x\right) = \frac{\area (\Omega)}{4\pi} x + \frac{ \length(\partial \Omega)}{4\pi} \sqrt{x} +O\Big(x^{1/3-\delta} \Big) .
\]
Applying Lemma~\ref{lem:dn and weyl} gives
\[
d_n(\sigma) =O\Big( n^{1/3-\delta} \Big)
\]
which proves Theorem~\ref{thm:upper bd for disk}. \qed

\subsection{Proof of Proposition~\ref{rel between weyl and lattice} }
 Fix a Robin parameter $\sigma\ge0$. Separating variables in polar coordinates $ (r,\theta) $ and inserting the boundary conditions, we find a basis of eigenfunctions of the form \[ f_{n,k}(r,\theta) = J_n(\kappa_{n,k} r)e^{i n \theta},\hspace{1em} n\in \mathbb{Z}, \; k=1,2,\dots\] with eigenvalues $\kappa_{n,k}^{2}$, where $\kappa_{n,k}$ is the $ k $-th positive zero of $xJ_{n}'\left(x\right)+\sigma J_{n}\left(x\right)$.
 In particular, for the Neumann case ($ \sigma=0 $), we get zeros of the derivative $ J'_n(x) $, denoted by $ j'_{n,k} $; since zero is a Neumann eigenvalue we use the standard convention that $ x=0 $ is counted as the first zero of $ J'_0(x) $.

Let
\[
S=\left\{ \left(x,y\right):y\ge\max\left(0,-x\right)\right\} ,
\]
and let $F:S\to\mathbb{R}$ be the degree $1$ homogeneous function
satisfying $F\equiv1$ on the graph of $g$. Obviously,
\[
F\left(n,k+\max\left(0,-n\right)-\frac{3}{4}\right)\le\mu\Longleftrightarrow\left(n,k+\max\left(0,-n\right)-\frac{3}{4}\right)\in\mu D;
\]
on the other hand, as will be shown in Lemma \ref{lem:GeomEvs} below, the numbers $\kappa_{n,k}$ are well approximated
by $F\left(n,k+\max\left(0,-n\right)-\frac{3}{4}\right)$.
This will give the desired connection between Weyl's law on the disk and the lattice count problem in dilations of $D$.

\begin{lem}
\label{lem:GeomEvs}Fix $\sigma\ge0$, and let $c>0$ be a constant.\\
1. As $n\to\infty$, uniformly for $k\le n/c$, we have
\begin{equation}
\kappa_{n,k}=F(n,k-\frac{3}{4})+O_{c,\sigma}\left(\frac{n^{1/3}}{k^{4/3}}\right).\label{eq:kappa_k<n}
\end{equation}
2. As $k\to\infty$, uniformly for $\left|n\right|\le c\cdot k$,
we have
\begin{equation}
\kappa_{n,k}=F\left(n,k+\max\left(0,-n\right)-\frac{3}{4}\right)+O_{c,\sigma}\left(\frac{1}{k}\right).\label{eq:kappa_k>n}
\end{equation}
\end{lem}

The proof of Lemma \ref{lem:GeomEvs} will be given in Appendix \ref{sec:BesselAppendix}.

It will be handy to derive an explicit formula for the function $F$,
which we will now do. Let $\zeta=\zeta\left(z\right)$ be the solution
to the differential equation
\begin{equation}
\left(\frac{\text{d}\zeta}{\text{d}z}\right)^{2}=\frac{1-z^{2}}{\zeta z^{2}}\label{eq:zeta_diff_eq}
\end{equation}
which for $z\ge1$ is given by
\begin{equation}
\frac{2}{3}\left(-\zeta\right)^{3/2}=\sqrt{z^{2}-1}-\arccos\left(\frac{1}{z}\right)\label{eq:zeta}
\end{equation}
(see \cite[Eq. 10.20.3]{DLMF}). The interval $z\ge1$ is bijectively
mapped to the interval $\zeta\le0$; denote by $z=z\left(\zeta\right)$
the inverse function.
\begin{lem}
For $x>0$, we have
\begin{equation}
F(x,y)=xz\Big(-x^{-2/3}\left(\frac{3\pi}{2}y\right)^{2/3}\Big).\label{eq:F_as_Z}
\end{equation}
Additionally, for $y\ge0$ we have $F\left(0,y\right)=\pi y$, and
for $\left(-x,y\right)\in S$ we have
\begin{equation}
F\left(-x,y\right)=F(x,y-x).\label{eq:F_symmetry}
\end{equation}
\end{lem}

\begin{proof}
Let $x>0$, and denote $t=\frac{F\left(x,y\right)}{x}$. Then $F\left(\frac{1}{t},\frac{y}{tx}\right)=1$
so that the point $\left(\frac{1}{t},\frac{y}{tx}\right)$ lies on
the graph of $g,$ and therefore
\[
\frac{y}{x}=\frac{1}{\pi}\left(\sqrt{t^{2}-1}-\arccos\left(\frac{1}{t}\right)\right)=\frac{1}{\pi}\frac{2}{3}\left(-\zeta\left(t\right)\right)^{3/2}
\]
so that
\[
t=z\Big(-x^{-2/3}\left(\frac{3\pi}{2}y\right)^{2/3}\Big).
\]
The other claims are also straightforward from the definitions.
\end{proof}
We proceed towards the proof of Proposition~\ref{rel between weyl and lattice} by following the ideas of \cite[Sec. 3]{CdV disk}. Let
\begin{align*}
N_{D}^{1}\left(\mu\right) & =\#\left\{ \left(n,k\right):\left(n,k+\max\left(0,-n\right)-\frac{3}{4}\right)\in\mu D,\:\left|n\right|<c\cdot k\right\} ,\\
N_{D}^{2}\left(\mu\right) & =\#\left\{ \left(n,k\right):\left(n,k-\frac{3}{4}\right)\in\mu D,\:n\ge c\cdot k\right\} ,
\end{align*}
and
\begin{align*}
N_{\text{disk},\sigma}^{1}\left(\mu^{2}\right) & =\#\left\{ \left(n,k\right):\kappa_{n,k}\le\mu,\,\left|n\right|<c\cdot k\right\} ,\\
N_{\text{disk},\sigma}^{2}\left(\mu^{2}\right) & =\#\left\{ \left(n,k\right):\kappa_{n,k}\le\mu,\,n\ge c\cdot k\right\} ,
\end{align*}
so that
\[
N_{D}\left(\mu\right)=N_{D}^{1}\left(\mu\right)+2N_{D}^{2}\left(\mu\right)
\]
and
\[
N_{\text{disk},\sigma}\left(\mu^{2}\right)=N_{\text{disk},\sigma}^{1}\left(\mu^{2}\right)+2N_{\text{disk},\sigma}^{2}\left(\mu^{2}\right),
\]
where we used (\ref{eq:F_symmetry}) and the relation $\kappa_{-n,k}=\kappa_{n,k}$.
We first compare  $N_{D}^{1}\left(\mu\right)$ and $N_{\text{disk},\sigma}^{1}\left(\mu^{2}\right)$:
\begin{lem}
\label{lem:LargeRange}There exists a constant $C=C_{c,\sigma}>0$
such that
\[
N_{D}^{1}\left(\mu-\frac{C}{\mu}\right)\le N_{\mathrm{disk},\sigma}^{1} \left(\mu^{2}\right)\le N_{D}^{1}\left(\mu+\frac{C}{\mu} \right).
\]
\end{lem}

\begin{proof}
Assume that $ |n|<c \cdot k $. By \eqref{eq:F_symmetry} and the homogeneity of $ F $ we have \[ F\left(n,k+\max(0,-n)-\frac{3}{4}\right)=F\left(|n|,k-\frac{3}{4}\right)=kF\left(\frac{|n|}{k},1-\frac{3}{4k}\right), \] and since $ 1\ll F\left(\frac{|n|}{k},1-\frac{3}{4k}\right) \ll_c 1 $, we conclude that  \[ k \ll F\left(n,k+\max(0,-n)-\frac{3}{4}\right) \ll_c k.\]
Hence, if $F\left(n,k+\max(0,-n)-\frac{3}{4}\right) \ge \mu $, then $ k \gg_c \mu $.
Combining this with Lemma \ref{lem:GeomEvs}, we see that
\[
\begin{split}
N_{\text{disk},\sigma}^{1}\left(\mu^{2}\right) & \le\#\left\{ \left(n,k\right):F(n,k+\max (0,-n  )-\frac{3}{4})\le\mu+\frac{C'}{k},\,\left|n\right|<c\cdot k \right\} \\
 & =\#\left\{ \left(n,k\right):F(n,k+\max\left(0,-n\right)-\frac{3}{4})\le\mu,\,\left|n\right|<c\cdot k\right\} \\
 & +\#\left\{ \left(n,k\right):\mu<F(n,k+\max\left(0,-n\right)-\frac{3}{4})\le\mu+\frac{C'}{k},\,\left|n\right|<c\cdot k\right\} \\
 & \le\#\left\{ \left(n,k\right):F(n,k+\max\left(0,-n\right)-\frac{3}{4})\le\mu+\frac{C}{\mu},\,\left|n\right|<c\cdot k\right\} \\
 & =N_{D}^{1}\left(\mu+\frac{C}{\mu}\right).
\end{split}
\]
The proof of the other inequality is similar.
\end{proof}
We will now compare between $N_{D}^{2}\left(\mu\right)$ and $N_{\text{disk},\sigma}^{2}\left(\mu^{2}\right)$.
To this end, for fixed $k\ge1$, we denote
\begin{align*}
N_{k}\left(\mu\right) & =\#\left\{ n:\left(n,k-\frac{3}{4}\right)\in\mu D,\,n\ge c\cdot k\right\} \\
N_{k}'\left(\mu\right) & =\#\left\{ n:\kappa_{n,k}\le\mu,\,\,n\ge c\cdot k\right\} .
\end{align*}

\begin{lem}
\label{lem:MainLemma}Given a sufficiently large $ c>0 $, there exists a constant $C=C_{c,\sigma}>0$
such that
\[
N_{k}\left(\mu\right)-C\frac{\mu^{1/3}}{k^{4/3}}-1\le N_{k}'\left(\mu\right)\le N_{k}\left(\mu\right)+C\frac{\mu^{1/3}}{k^{4/3}}+1.
\]
\end{lem}

\begin{proof}
Let
\[
A_{k}\left(\mu\right):=\#\left\{ n:\,\mu<F\left(n,k-\frac{3}{4}\right)\le\mu+C'\frac{\mu^{1/3}}{k^{4/3}},\,n\ge c\cdot k\right\},
\]and recall the inequality (see \eqref{eq:interlacing}) $n \le j'_{n,k } \le \kappa_{n,k} $, so in particular if $ \kappa_{n,k} \le \mu $, then $ n\le \mu $.
Thus, Lemma \ref{lem:GeomEvs} gives
\begin{align*}
N_{k}'\left(\mu\right) & \le\#\left\{ n:F (n,k-\frac{3}{4} )\le\mu+C'\frac{n^{1/3}}{k^{4/3}},\,\mu \ge n\ge c\cdot k\right\} \\
 &  \le\#\left\{ n:F (n,k-\frac{3}{4} )\le\mu,\,n\ge c\cdot k\right\} +A_{k}\left(\mu\right)\\
 & =N_{k}\left(\mu\right)+A_{k}\left(\mu\right).
\end{align*}
When $x\ge c \cdot k$, we have $F\left(x,k-\dfrac{3}{4}\right)=xF\left(1,\dfrac{k-3/4}{x}\right)$, and therefore (note that $ F(1,y)\ge 1 $ for all $ y\ge 0 $)
\begin{align*} F_x\left(x,k-\dfrac{3}{4}\right) = F\left(1,\dfrac{k-3/4}{x}\right) -\dfrac{k-3/4}{x}F_y\left(1,\dfrac{k-3/4}{x}\right) \gg 1
\end{align*} when $ c $ is taken sufficiently large.
In particular, $ \tilde{F}(x) := F(x,k-\dfrac{3}{4}) $ is strictly increasing for $ x\ge c \cdot k $, and so  $A_k(\mu)$ is bounded above by the number of integer points in the interval $$I := \left[\tilde{F}^{-1}\left(\max(\mu,\tilde{F}(c\cdot k))\right), \tilde{F}^{-1}\left(\mu+C'\frac{\mu^{1/3}}{k^{4/3}}\right)\right],$$ which in turn is bounded above by $ \mathrm{length}(I) + 1$; by the mean value theorem, keeping in mind that $ (\tilde{F}^{-1})_x = \tilde{F}_x^{-1},$ we conclude that
\[
\mathrm{length}(I) \le C'\frac{\mu^{1/3}}{k^{4/3}} \cdot \max_{x\in I} \dfrac{1}{\tilde{F}_x(x)} \le C \frac{\mu^{1/3}}{k^{4/3}}.
\]
The proof of the other inequality is similar.
\end{proof}
\begin{remark}
The $+1$ factor was missing in \cite{CdV disk}.
\end{remark}

For large values of $k$ we will use the following estimate:
\begin{lem}
\label{lem:IntermediateRange}There exists a constant $C=C_{c,\sigma}>0$
such that for $k>\mu^{4/7}$, we have
\[
N_{k}\left(\mu-\frac{C}{\mu^{3/7}}\right)\le N_{k}'\left(\mu\right)\le N_{k}\left(\mu+\frac{C}{\mu^{3/7}}\right).
\]
\end{lem}

\begin{proof}
By Lemma \ref{lem:GeomEvs},
\[
N_{k}'\left(\mu\right)\le\#\left\{ n:F\left(n,k-\frac{3}{4}\right)
\le\mu+\frac{C}{\mu^{3/7}},\,n\ge c\cdot k\right\}
=N_{k}\left(\mu+\frac{C}{\mu^{3/7}}\right)  
\]
and likewise
\[
N_{k}'\left(\mu\right)\ge N_{k}\left(\mu-\frac{C}{\mu^{3/7}}\right).
\]
\end{proof}
\begin{proof}[Proof of Proposition~\ref{rel between weyl and lattice}]
By Lemma \ref{lem:MainLemma} (applied for $k\le\mu^{4/7}$) and
Lemma \ref{lem:IntermediateRange} (applied for $k>\mu^{4/7}$), we
get that
\begin{multline*}
N_{\text{disk},\sigma}^{2}\left(\mu^{2}\right)=\sum_{k\ge1}N_{k}'\left(\mu\right)
\\
\le\sum_{k\ge1}N_{k}\left(\mu+\frac{C}{\mu^{3/7}}\right)+C\mu^{4/7}
=N_{D}^{2}\left(\mu+\frac{C}{\mu^{3/7}}\right)+C\mu^{4/7}
\end{multline*}
and likewise
\[
N_{\text{disk},\sigma}^{2}\left(\mu^{2}\right)\ge N_{D}^{2}\left(\mu-\frac{C}{\mu^{3/7}}\right)-C\mu^{4/7}.
\]
This, together with Lemma \ref{lem:LargeRange} gives the claim.
\end{proof}

\appendix

\section{\label{sec:BesselAppendix}Proof of Lemma \ref{lem:GeomEvs}}

The goal of this appendix is to prove the asymptotic formulas (\ref{eq:kappa_k<n})
and (\ref{eq:kappa_k>n}) for the zeros $\kappa_{n,k}$ of $xJ_{n}'\left(x\right)+\sigma J_{n}\left(x\right)$
where $\sigma\ge0.$ More generally, we will work with Bessel functions
$J_{\nu}\left(x\right)$ of real order $\nu$. Many properties of
the zeros $\kappa_{\nu,k}$ are well-known, e.g. for all $\sigma>0$,
$\nu \ge 0$ and $k\ge1$ we have (see e.g. \cite[Eq. (III.6)]{Spigler})
\begin{equation}
\nu \le j'_{\nu,k}<\kappa_{\nu,k}<j_{\nu,k}\label{eq:interlacing}
\end{equation}
where $ j_{\nu,k} $ (resp. $ j'_{\nu,k} $) is the $ k $-th positive zero of $ J_{\nu}\left(x\right) $ (resp. $ J'_{\nu}\left(x\right) $, with the convention that $ x=0  $ is counted as the first zero of $ J'_0(x) $);
for $\sigma\ge0$ and fixed $\nu \ge 0$ we have the asymptotic formula
\cite[Eq. (IV.9)]{Spigler}
\[
\kappa_{\nu,k}=j'_{\nu,k}+\frac{\sigma}{j'_{\nu,k}}+\frac{-\frac{1}{3}\sigma^{3}-\frac{1}{2}\sigma^{2}+\nu^{2}\sigma}{\left(j'_{\nu,k}\right)^{3}}+O_{\sigma}\left(\left(j'_{\nu,k}\right)^{-5}\right)\hspace{1em}\left(k\to\infty\right).
\]
Recall the function $\zeta\left(z\right)$ defined above by (\ref{eq:zeta_diff_eq})
which satisfies (\ref{eq:zeta}) for $z\ge1$, with an inverse $z\left(\zeta\right)$.
Denote $h\left(\zeta\right)=\left(\frac{4\zeta}{1-z^{2}}\right)^{1/4}$.
We have the following asymptotic expansion for $J_{\nu}\left(\nu z\right)$
as $\nu\to\infty$ \cite[Eq. 10.20.4]{DLMF}
\begin{equation}
J_{\nu}\left(\nu z\right)\sim h\left(\zeta\right)\left[\frac{\text{Ai}\left(\nu^{2/3}\zeta\right)}{\nu^{1/3}}\sum_{j=0}^{\infty}\frac{A_{j}\left(\zeta\right)}{\nu^{2j}}+\frac{\text{Ai}'\left(\nu^{2/3}\zeta\right)}{\nu^{5/3}}\sum_{j=0}^{\infty}\frac{B_{j}\left(\zeta\right)}{\nu^{2j}}\right]\label{eq:J_n_asymp}
\end{equation}
which holds uniformly for $ z>0, $ where $\text{Ai\ensuremath{\left(z\right)}}$ is the Airy function,
and the coefficients $A_{j}\left(\zeta\right)$ and $B_{j}\left(\zeta\right)$
are given by \cite[Eq. 10.2.10, 10.20.11]{DLMF} and the remark following
these equations. Likewise, we have the asymptotic expansion \cite[Eq. 10.20.7]{DLMF}
\begin{equation}
J_{\nu}'\left(\nu z\right)\sim-\frac{2}{zh\left(\zeta\right)}\left[\frac{\text{Ai}\left(\nu^{2/3}\zeta\right)}{\nu^{4/3}}\sum_{j=0}^{\infty}\frac{C_{j}\left(\zeta\right)}{\nu^{2j}}+\frac{\text{Ai}'\left(\nu^{2/3}\zeta\right)}{\nu^{2/3}}\sum_{j=0}^{\infty}\frac{D_{j}\left(\zeta\right)}{\nu^{2j}}\right]\label{eq:J_n'_asymp}
\end{equation}
uniformly for $ z>0, $ where the coefficients $C_{j}\left(\zeta\right)$ and $D_{j}\left(\zeta\right)$
are given by \cite[Eq. 10.2.12, 10.20.13]{DLMF} and the remark which
follows them. Each of the coefficients $A_{j}\left(\zeta\right),B_{j}\left(\zeta\right),C_{j}\left(\zeta\right),D_{j}\left(\zeta\right)$,
$j=0,1,2,\dots$ is bounded near $\zeta=0$; we have $A_{0}\left(\zeta\right)=D_{0}\left(\zeta\right)=1$.

For the Robin parameter $\sigma\ge0$, if we denote $B_{-1}\left(\zeta\right)=0$
and let
\begin{align*}
	\alpha_{j}^{\sigma}\left(\zeta\right) & :=C_{j}\left(\zeta\right)-\frac{\sigma A_{j}\left(\zeta\right)h^{2}\left(\zeta\right)}{2}\\
	\beta_{j}^{\sigma}\left(\zeta\right) & :=D_{j}\left(\zeta\right)-\frac{\sigma B_{j-1}\left(\zeta\right)h^{2}\left(\zeta\right)}{2},
\end{align*}
then (\ref{eq:J_n_asymp}) and (\ref{eq:J_n'_asymp}) give
\begin{align*}
\phi_{\nu}\left(\nu z\right)&:=J_{\nu}'\left(\nu z\right)+\frac{\sigma}{\nu z}J_{\nu}\left(\nu z\right)\\
&\sim\frac{-2}{zh\left(\zeta\right)}\left[\frac{\text{Ai}'\left(\nu^{2/3}\zeta\right)}{\nu^{2/3}}\sum_{j=0}^{\infty}\frac{\beta_{j}^{\sigma}\left(z\right)}{\nu^{2j}}+\frac{\text{Ai}\left(\nu^{2/3}\zeta\right)}{\nu^{4/3}}\sum_{j=0}^{\infty}\frac{\alpha_{j}^{\sigma}\left(z\right)}{\nu^{2j}}\right]
\end{align*}
uniformly for $ z>0. $ Note that $\alpha_{0}^{\sigma}\left(\zeta\right)=C_{0}\left(\zeta\right)-\frac{\sigma h^{2}\left(\zeta\right)}{2}$.
Using the derivation of \cite[p. 345]{Olver} with $\alpha_{j}^{\sigma}$,
$\beta_{j}^{\sigma}$ instead of $C_{j}\left(\zeta\right)$, $D_{j}\left(\zeta\right)$
(the latter were used to establish the asymptotic expansion of the
zeros of $J_{\nu}'\left(z\right)$ corresponding to $\sigma=0$),
we get the following \emph{uniform} asymptotic formula for $\kappa_{\nu,k}$
as $\nu\to\infty$ :
\begin{lem}
	Fix $\sigma\ge0$, let $a_{k}'$ be the $k$-th zero of $\mathrm{Ai}'\left(z\right)$
	(all of these zeros are real and negative), and let $\zeta=\zeta_{\nu,k}=\nu^{-2/3}a_{k}'$.
	Then in the above notation, uniformly for $ k\ge1 $, we have
	\begin{equation}
	\kappa_{\nu,k}=\nu z\left(\zeta\right)-\frac{z'\left(\zeta\right)\left(C_{0}\left(\zeta\right)-\frac{\sigma h^{2}\left(\zeta\right)}{2}\right)}{\zeta\nu}+O_{\sigma}\left(\frac{1}{\nu}\right).\label{eq:k_nm}
	\end{equation}
\end{lem}

In particular, for $\sigma=0$ we reconstructed the formula \cite[Eq. 10.21.43]{DLMF}
\[
j'_{\nu,k}=\nu z\left(\zeta\right)-\frac{z'\left(\zeta\right)C_{0}\left(\zeta\right)}{\zeta\nu}+O\left(\frac{1}{\nu}\right)
\]
uniformly for $ k\ge1 $ (note the identity $z'\left(\zeta\right)=-\frac{z\left(\zeta\right)h^{2}\left(\zeta\right)}{2}$).
We remark that the secondary term in (\ref{eq:k_nm}) is necessary
because of the $\zeta$ factor in the denominator which may be as
small as $\nu^{-2/3}$ when $k$ is small. This phenomenon does not
occur for the zeros of $J_{\nu}\left(z\right)$ (see \cite[Sec. 7]{Olver}),
which satisfy the more compact uniform expansion
\[
j_{\nu,k}=\nu z\left(\zeta\right)+O\left(\frac{1}{\nu}\right)
\]
where $\zeta=\nu^{-2/3}a_{k}$ and $a_{k}$ is the $k$-th zero of
$\text{Ai}\left(z\right)$ \cite[Eq. 10.21.41]{DLMF}.

Recall that the zeros $a_{k}'$ of $\text{Ai}'\left(z\right)$ satisfy
the asymptotic formula \cite[Eq. 9.9.8]{DLMF}
\begin{equation}
a_{k}'=-\left[\frac{3\pi}{2}\left(k-\frac{3}{4}\right)\right]^{2/3}+O\left(k^{-4/3}\right).\label{eq:AiryZeros}
\end{equation}

\begin{proof}[Proof of Lemma \ref{lem:GeomEvs}, first part]
	Assume that $k\le\nu/c$, where $\nu\to\infty$. The functions $z'\left(\zeta\right)$
	and $h\left(\zeta\right)$ are bounded near $\zeta=0$, and therefore
	inserting (\ref{eq:AiryZeros}) into (\ref{eq:k_nm}) gives
	\[
	\nu z\left(\zeta\right)=\nu z\left(-\nu^{-2/3}\left[\frac{3\pi}{2}\left(k-\frac{3}{4}\right)\right]^{2/3}\right)+O_{c,\sigma}\left(\frac{\nu^{1/3}}{k^{4/3}}\right)
	\]
	and
	\[
	\frac{z'\left(\zeta\right)\left(C_{0}\left(\zeta\right)-\frac{\sigma h^{2}\left(\zeta\right)}{2}\right)}{\zeta\nu}\ll_{c,\sigma}\frac{1}{\nu^{1/3}k^{2/3}}\ll_c\frac{\nu^{1/3}}{k^{4/3}}.
	\]
	Also note that $\frac{1}{\nu}\ll_c\frac{\nu^{1/3}}{k^{4/3}}$ when $k\le\nu$/c.
	By (\ref{eq:F_as_Z}) we have
	\[
	z\left(-\nu^{-2/3}\left[\frac{3\pi}{2}\left(k-\frac{3}{4}\right)\right]^{2/3}\right)=F\left(\nu,k-\frac{3}{4}\right),
	\]
	and therefore the above estimates yield
	\begin{equation}
	\kappa_{\nu,k}=F\left(\nu,k-\frac{3}{4}\right)+O_{c,\sigma}\left(\frac{\nu^{1/3}}{k^{4/3}}\right)\label{eq:kappa_nu}
	\end{equation}
	which gives (\ref{eq:kappa_k<n}).
\end{proof}
In order to prove the second part of Lemma \ref{lem:GeomEvs}, we
require the following lemma. Recall the function $g\left(x\right)$
defined in (\ref{eq:g_def}).
\begin{lem}
	\label{lem:phi_asymp_lemma}Fix $\sigma\ge0$, and let $C>0$ be a
	constant. Let $\phi_{\nu}\left(x\right):=J_{\nu}'\left(x\right)+\frac{\sigma}{x}J_{\nu}\left(x\right)$.
	As $x\to\infty$, uniformly for $0\le\nu\le x/\left(1+C\right)$,
	we have
	\begin{equation}
	\phi_{\nu}\left(x\right)=-\left(\frac{2}{\pi}\right)^{1/2}\left(x^{2}-\nu^{2}\right)^{1/4}x^{-1}\left(\sin\left(\pi xg\left(\nu/x\right)-\frac{\pi}{4}\right)+O_{C,\sigma}\left(x^{-1}\right)\right).\label{eq:phi_asymp}
	\end{equation}
\end{lem}

\begin{proof}
	We use the standard integral representation  \cite[Eq. 10.9.6]{DLMF}
	\begin{equation}
	J_{\nu}\left(x\right)=\frac{1}{\pi}\int_{0}^{\pi}\cos\left(x\sin t-\nu t\right)\,dt-\frac{\sin\left(\nu\pi\right)}{\pi}\int_{0}^{\infty}e^{-x\sinh t-\nu t}\,dt\label{eq:BesselIntegral}
	\end{equation}
	(for integer $\nu$ the second integral in (\ref{eq:BesselIntegral})
	vanishes). Assume that $x\ge\left(1+C\right)\nu\ge0$, and denote
	$r=\nu/x\le\frac{1}{1+C}<1$. The first integral in (\ref{eq:BesselIntegral})
	is equal to the real part of
	\[
	\mathcal{I}_{r}^{1}\left(x\right):=\frac{1}{\pi}\int_{0}^{\pi}e^{ix\left(\sin t-rt\right)}\,dt.
	\]
	By the method of stationary phase \cite[Eq. 2.3.23]{DLMF}, we have
	the asymptotics:
	\begin{align*}
	\mathcal{I}_{r}^{1}\left(x\right)&=e^{\pi i\left(xg\left(r\right)-\frac{1}{4}\right)}\left(\frac{2}{\pi\sqrt{1-r^{2}}x}\right)^{1/2}\\
	&+e^{-ir\pi x}\frac{i}{\pi\left(1+r\right)x}+\frac{i}{\pi\left(1-r\right)x}+O_{C}\left(x^{-3/2}\right).
	\end{align*}
	Hence
	\begin{align*}
	\text{Re}\left(\mathcal{I}_{r}^{1}\left(x\right)\right)&=\cos\left(\pi xg\left(r\right)-\frac{\pi}{4}\right)\left(\frac{2}{\pi\sqrt{1-r^{2}}x}\right)^{1/2}\\
	&+\frac{\sin\left(\pi rx\right)}{\pi\left(1+r\right)x}+O_{C}\left(x^{-3/2}\right).
	\end{align*}
	The second integral in (\ref{eq:BesselIntegral}) is equal to
	\[
	\mathcal{I}_{r}^{2}\left(x\right):=\frac{\sin\left(\pi rx\right)}{\pi}\int_{0}^{\infty}e^{-x\left(\sinh t+rt\right)}\,dt
	\]
	and can be evaluated by the Laplace method \cite[Eq. 2.3.15]{DLMF}:
	\[
	\mathcal{I}_{r}^{2}\left(x\right)=\frac{\sin\left(\pi rx\right)}{\pi\left(1+r\right)x}+O_{C}\left(x^{-2}\right).
	\]
	We obtain
	\begin{equation}
	J_{\nu}\left(x\right)=\left(\frac{2}{\pi}\right)^{1/2}\left(x^{2}-\nu^{2}\right)^{-1/4}\left(\cos\left(\pi xg\left(\nu/x\right)-\frac{\pi}{4}\right)+O_{C}\left(x^{-1}\right)\right);\label{eq:j_nu_asymp}
	\end{equation}
	a similar procedure gives
	\begin{equation}
	J_{\nu}'\left(x\right)=-\left(\frac{2}{\pi}\right)^{1/2}\left(x^{2}-\nu^{2}\right)^{1/4}x^{-1}\left(\sin\left(\pi xg\left(\nu/x\right)-\frac{\pi}{4}\right)+O_{C}\left(x^{-1}\right)\right).\label{eq:j'_nu_asymp}
	\end{equation}
	The formula (\ref{eq:phi_asymp}) now follows upon combining (\ref{eq:j_nu_asymp})
	and (\ref{eq:j'_nu_asymp}).
\end{proof}
\begin{proof}[Proof of Lemma \ref{lem:GeomEvs}, second part]
	Let $0\le\nu\le c\cdot k,$ where $k\to\infty$. Clearly, the condition
	$0\le\nu\le c\cdot k$ implies that $\kappa_{\nu,k}\ge\left(1+C\right)\nu\ge0$
	for some constant $C=C\left(c\right)>0$ and that $ \kappa_{\nu,k} \asymp_c k $ (e.g. by the analogous well-known inequalities
	for the Bessel zeros $j_{\nu,k}$, see (5.3) in \cite{CdVGJ}, together with (\ref{eq:interlacing})).
	By Lemma \ref{lem:phi_asymp_lemma}, we have
	\[
	\sin\left(\pi \kappa_{\nu,k} g\left(\nu/\kappa_{\nu,k}\right)-\frac{\pi}{4}\right)=O_{c, \sigma}\left(k^{-1}\right)
	\]
	so there exists an integer $m$ such that
	\begin{equation}
    \kappa_{\nu,k}g\left(\nu/\kappa_{\nu,k}\right)=m-\frac{3}{4}+O_{c,\sigma}\left(k^{-1}\right).\label{eq:xg(nu/x)}
	\end{equation}
    This in particular gives $ m \asymp_{c,\sigma} \kappa_{\nu,k}$. We have 1 $ \ll F_{y} \ll_{c,\sigma} 1 $ when $ y\gg_{c,\sigma} x $, as can be seen by differentiating \eqref{eq:F_as_Z} and combining with \eqref{eq:zeta_diff_eq}. This, together with the equality $ \kappa_{\nu,k} = F(\nu, \kappa_{\nu,k}g(\nu / \kappa_{\nu,k})) $ and (\ref{eq:xg(nu/x)}) gives
	\begin{equation*}
	\kappa_{\nu,k}=F\left(\nu,m-\frac{3}{4}+O_{c,\sigma}\left(k^{-1}\right)\right)=F\left(\nu,m-\frac{3}{4}\right)+O_{c,\sigma}\left(k^{-1}\right).
	\end{equation*}
	We will now show that $ m=k $: indeed, fix $ k, \nu $ and the corresponding $ m $, and assume that $ \nu' $ is close to $ \nu $, so there exists an integer $ m' $ such that \[\kappa_{\nu',k}=F\left(\nu',m'-\frac{3}{4}\right)+O_{c,\sigma}\left(k^{-1}\right). \] By the mean value theorem
	\begin{align}\nonumber
		|m-m'| & \ll \left\vert F\left(\nu',m-\frac{3}{4}\right) - F\left(\nu',m'-\frac{3}{4}\right)\right\vert \\ \label{eq:cont_argument}
		& \le \left\vert F\left(\nu',m-\frac{3}{4}\right) - F\left(\nu,m-\frac{3}{4}\right)\right\vert  + |\kappa_{\nu',k}-\kappa_{\nu,k}| \\ &+O_{c,\sigma}\left(k^{-1}\right)  \nonumber.
	\end{align}
Note that $\kappa_{\nu',k}$ is a continuous function of $\nu'\ge0$
	(in fact, it is differentiable in $\nu'$ in this regime, see e.g.
	\cite{LandauLJ}), and so is $ F\left(\nu',m-\frac{3}{4}\right) $ as a function of $ \nu' $. Therefore the right-hand-side of \eqref{eq:cont_argument} can be made arbitrarily small when $ k $ is sufficiently large and $ \nu'$ is sufficiently close to $ \nu $, and hence $ m'=m $. We see that map $ \nu \mapsto \kappa_{\nu,k} \mapsto m $ is well-defined and it is locally constant and hence constant for $ 0 \le \nu \le c \cdot k $.  But we know by (\ref{eq:kappa_nu}) that for $\nu\asymp k$ we have $ m=k $; hence $ m=k $ for any $ 0 \le \nu \le c \cdot k $. This gives (\ref{eq:kappa_k>n}) when $n\ge0;$
	for $n<0$, (\ref{eq:kappa_k>n}) follows from the relations $\kappa_{-n,k}=\kappa_{n,k}$
	and (\ref{eq:F_symmetry}).
\end{proof}

\end{document}